\documentclass[11pt]{article}
\usepackage{graphicx}
\usepackage{amsmath, amsthm,latexsym,amsbsy,amssymb, color,mathtools}
\usepackage{graphicx,pgf,tikz, enumitem}
\usetikzlibrary{arrows}

\usepackage{graphicx}
\font\bigbf=cmbx10 at 16pt

\def\ds{\displaystyle}
\def\forall{\hbox{for all}~}

\def\R{\mathbb R}

\def\vsk{\vskip 4em}
\def\v{\vskip 1em}

\def\sqr#1#2{\vbox{\hrule height .#2pt
\hbox{\vrule width .#2pt height #1pt \kern #1pt
\vrule width .#2pt}\hrule height .#2pt }}

\newcommand{\set}[1]{\left\{#1\right\}}

\newcommand{\p}[1]{\left( #1 \right)}

\newtheorem{theorem}{Theorem}[section]
\newtheorem{definition}[theorem]{Definition}
\newtheorem{lemma}[theorem]{Lemma}

\newtheorem{remark}[theorem]{Remark}
\newtheorem{proposition}[theorem]{\textbf{Proposition}}
\begin{document}

\title{\bigbf Solutions to a system of first order H-J equations related to a debt management problem 
}

\vsk

\author{Antonio Marigonda$^{(1)}$ and Khai T. Nguyen$^{(2)}$
\\
 {\small $^{(1)}$ Department of Computer Science, University of Verona, ITALY}\\  
  {\small $^{(2)}$ Department of Mathematics, North Carolina State University, USA}\\ 
 \\  {\small E-mails: ~antonio.marigonda@univr.it,  ~ khai@math.ncsu.edu}
}

\maketitle

\begin{abstract}
The paper studies a system of first order Hamilton-Jacobi equations with discontinuous coefficients,  arising from a model of deterministic optimal debt management in infinite time horizon, with exponential discount and currency devaluation. The existence of an equilibrium solution is obtained by a suitable concatenation of backward solutions to the system of Hamilton-Jacobi equations. A detailed analysis of the behavior of the solution as the debt-ratio-income $x^*\to +\infty$ is also provided. 
\vspace{0.2in}

\noindent
{\bf Key  words.} Hamilton-Jacobi equations, optimal debt management, equilibrium solutions
\end{abstract}

\section{Introduction} Consider a system of Hamilton-Jacobi equation 
\begin{equation}\label{HJ0}
    \begin{cases}
        rV~ =~\ds H(x,V',p) + \dfrac{\sigma^2x^2}{2}\cdot V'', \\[10pt]

        \p{ r + \lambda + v(x)}\cdot p - \p{ r + \lambda} ~=~H_\xi(x, V', p) \cdot p' +\dfrac{\sigma^2x^2}{2}\cdot p'',\\[10pt]

  v(x)~=~\mathrm{argmin}_{w \geq 0 } \set{ c(v) - wx V'(x)} \,,
    \end{cases}
\end{equation}
with the boundary conditions
\[
V(0)~=~0,\qquad V(x^*)~=~B\qquad \textrm{and }\qquad p(0)~=~1,\qquad p(x^*)~=~\theta(x^*),
\]
motivated by an optimal debt management problem in infinite time horizon with exponential discount.  As in \cite{BN, BJ, BMNP,  CGN, AK, NT}, this modeled as a noncooperative interaction between a borrower and a pool of risk-neutral lenders. Here, the independent variable $x$ is the debt-to-income ratio, $x^*$ is a threshold of the  debt-to-income ratio where the borrower  must declare bankruptcy,  the salvage function $\theta\in [0,1]$ determines the fraction of capital that can be recovered by lenders when bankruptcy occurs, and 
\begin{itemize}
\item $V$ is the value function for the
borrower who is a sovereign state that can decide the devaluation rate of its currency $v$  and the fraction of its income $u$ which is used to repay the debt,
\item $p$ is the  discounted rate at which the lenders buy bonds to offset the possible loss of part of their investment.
\end{itemize} 
Since $p$ is determined by the expected evolution of the  debt-to-income ratio at all future times, it depends globally on the entire feedback controls $u$ and $v$. This leads to a highly nonstandard optimal control problem, and  a ``solution'' must be understood as a Nash equilibrium, where the strategy implemented by the borrower represents the best reply to the strategy adopted by the lenders, and conversely.    
\medskip

For the stochastic model ($\sigma>0$), the authors proved in \cite{AK} the existence of an equilibrium solution  $(V_{\sigma},p_{\sigma})$ as a steady state of an auxiliary parabolic system. The proof requires a careful analysis to construct an invariant domain and apply a fixed-point result to derive the existence of a steady state for the auxiliary parabolic system. Moreover, they also established the upper (lower) bound of discounted bond price $p_{\sigma}$ and the expected total optimal cost for servicing the debt $V_{\sigma}$. Here, a natural question is  trying to understand whether a solution  exists and  its structure remains unchanged in the deterministic case ($\sigma=0$). A classical approach for a solution in this case is the vanishing viscosity method. More precisely, one studies the  limit $(V_\sigma,p_\sigma)\to (V,p)$ as  the diffusion coefficient $\sigma\to 0+$ and show that this limit   the limits $(V_{\sigma},p_{\sigma})$ yields a solution to  (\ref{HJ0}) with $\sigma=0$. However, this is a highly  nontrivial problem and still remains open. 
\medskip

The present paper aims to provide a direct study  to the deterministic case of the system (\ref{HJ0})  by looking at the  corresponding system of differential inclusions
\begin{equation}\label{DI}
\begin{cases}
V'(x)&\in~\left\{F^-(x,V(x),p(x)), F^+(x,V(x),p(x))\right\} \\
p'(x)&\in~\left\{G^-(x,V(x),p(x)), G^+(x,V(x),p(x))\right\}
\end{cases}
\end{equation}
where $F^{\pm}(x,V(x),p(x))$ solves the equation $rV(x)= H(x,\xi,p(x))$ with variable $\xi$, and 
\[
G^{\pm}(x,\eta,p)~=~\dfrac{(r+\lambda+v^*(x,F^{\pm}(x,\eta,p)))p-(r+\lambda)}{H_{\xi}(x,F^{\pm}(x,\eta,p),p)}.
\]
In Theorem \ref{thm:main2}, we first construct a solution $(V,p)$ of (\ref{DI}) with boundary conditions by a suitable concatenation of backward solutions, and then determine an equilibrium solution  to the corresponding differential game with deterministic dynamics. Moreover, we show that there exists a semi-equilibrium point $x_1\in [0,x^*[$  such that if the debt-ratio-income less than $x_1$, then the optimal strategy will reach a
steady state, otherwise bankruptcy in finite time is unavoidable. In our construction, the main technical difficulties in the analysis stem from the fact that, the system (\ref{DI}) is not monotone and $F^{\pm}$ are just H\"older continuous at points where $H_{\xi}$ vanishes. Moreover, $p(\cdot)$ may well have many discontinuities $x_k$. At these points,  backward solutions is not necessarily unique and does not a detail analysis. Thereafter,  in Proposition \ref{Pz}, using the the analysis of sub- and super-solutions, we study in an asymptotic behaviour of $(V,p)$ as the maximum debt-to-income threshold $x^*$ is pushed to $+\infty$. Consequently,
\begin{itemize}
\item if the salvage rate decay sufficiently slowly, i.e., the lenders can still recover a sufficiently high fraction of their investment after the bankruptcy, then the best choice for the borrower is to implement the Ponzi's scheme;
\medskip
\item otherwise, if the salvage rate $\theta(x^*)$ decays sufficiently fast, then Ponzi's scheme is no longer an optimal solution for the borrower;
\item for sufficiently large initial debt-to-income and bankruptcy threshold and recovery fraction after bankruptcy, the optimal strategy for the borrower will use currency devaluation $v$ to deflate the debt-to-income.
\end{itemize}
The remainder of the paper is organized as follows. In Section 2, we provide a more detailed description of the model and the  system of Hamilton-Jacobi equation satisfied by $(V,p)$, and study basic properties of  $H$. In Section 3, we  construct a solution to (\ref{HJ0}) with $\sigma=0$, and then derive an equilibrium solutions to the model of optimal debt management. In Section 4,  we  perform  a detailed analysis of the behavior of the optimal feedback controls  as $x^*\to\infty$. We close by an appendix which contains some concepts of convex analysis and collect some further technical results related to the Hamiltonian function.
\section{Model derivation and system of Hamilton-Jacobi equations} 
\setcounter{equation}{0}
\subsection{A deterministic optimal debt management problem} 
In this subsection, we shall recall our deterministic optimal debt management problem with currency devaluation with exponential discount in \cite{AK}. Here,  the borrower is a sovereign state, that can decide to devaluate its currency, and its total income $Y(t)$ and total debt $X(t)$ are  governed by the control dynamics
\begin{equation}\label{XY-system}
\begin{cases}
\dot{Y}(t)&=~(\mu +v(t)) Y(t),\cr\cr
\dot{X}(t)&=~-\lambda X(t) + \dfrac{(\lambda + r) X(t) - U(t)}{p(t)},
\end{cases} 
\end{equation}
where $\mu$ is the average growth rate of the economy, $\lambda$ is the rate at which the borrower pays back the principal, $r$ is the interest rate paid on bonds, and 
\begin{itemize}
\item $U(t)$ is  the rate of payments that the borrower chooses to make to the lenders at time $t$;
\item $v(t)\geq 0$ is the devaluation rate at time $t$, regarded as an additional control;
\end{itemize}
We define the debt-to-income ratio $x\doteq \ds{X\over Y}$, and set $u\doteq\ds {U\over Y}$. The system (\ref{XY-system}) yields 
\begin{equation}\label{eq:debt-GDP-ev}
\dot{x}(t)~=~\left(\dfrac{\lambda+r}{ p(t)}-\lambda-\mu- v(t)\right)x(t)-\dfrac{u(t)}{p(t)}.
\end{equation}
In this model, the borrower is forced to declare bankruptcy when the debt-ratio-income $x$ reaches threshold $x^*$. The bankruptcy time is denoted by 
\begin{equation}\label{Tb}
T_b~\doteq~\inf\{t>0: x(t)=x^*\}~\in~\R\cup \{+\infty\}.
\end{equation}
Without the presence of the devaluation of currency ($v=0$), when a foreign  investor buys a bond of unit nominal value, he will receive a continuous stream of
payments with intensity $(r+\lambda)e^{-\lambda t}$. If bankruptcy never occurs, the payoff for a foreign investor (exponentially discounted in time) is
\[
\Psi~=~\int_{0}^{\infty}e^{-rt}\cdot (r+\lambda)e^{-\lambda t}dt~=~1.
\]
Otherwise, the lenders recover only a fraction $\theta(x^*)\in[0,1]$ of their outstanding capital. In this case, taking account of  the presence of  the devaluation of currency, the pay of for a foreign investor will be
 \[
 \Psi~=~\int_{0}^{T_b}(r+\lambda)\cdot \exp\left\{\ds-\int_{0}^{t}(r+\lambda+v(s))ds\right\}dt+\exp\left\{\ds-\int_{0}^{T_b}(r+\lambda+v(s))ds\right\}\cdot \theta(x^*).
 \]
 If the outstanding capital is recovered in full (i.e., $\theta(x^*)=1$) and $v=0$, then again $\Psi=1$. In general, however, $\theta(x^*)<1, v\neq 0$, and thus $\Psi<1$. To offset this possible loss,  the investors buy a bond with unit nominal value at a discounted price
 \begin{equation}\label{p}
 p(t)~=~\int_{0}^{T_b}(r+\lambda)\cdot \exp\left\{\ds-\int_{0}^{t}(r+\lambda+v(s))ds\right\}dt+\exp\left\{\ds-\int_{0}^{T_b}(r+\lambda+v(s))ds\right\}\cdot \theta(x^*).
 \end{equation}
  Given an initial size $x_0$ of the debt-to-income ratio, the borrower wants to find a pair of optimal controls $(u,v)$ which minimizes his total expected cost, exponentially discounted in time:
 \begin{equation}\label{J}
\textrm{ minimize~} J~\doteq~\int_0^{T_b} e^{-rt}[L(u(t))+c(v(t))]\, dt + e^{-r T_b} B
\end{equation}
where $c(v)$ is the social cost resulting from devaluation, $L(u)$ is the cost to the borrower for putting income towards paying the debt, and $B$ is the cost of bankruptcy.
Throughout the paper we shall assume the following structural conditions on the cost functions $L$, $c$:
 \begin{enumerate}
\item[\textbf{(A1)}] \emph{The implementing cost function $L$ is twice continuously differentiable for  $u\in [0,1[$, and satisfies}
\[L(0)~=~0,\quad  L'(u)~>~0, \quad L''(u)~>~0~~\textrm{and }\lim_{u\to 1^-} L(u)~=~+\infty.\]
\item[\textbf{(A2)}] \emph{The social cost $c$ is twice continuously differentiable for $v\in [0,+\infty[$,and satisfies}
\[
c(0)~=~0,\quad c'(v)~>~0, \quad c''(v)~>~0~~\mathrm{and}~~\lim_{v\to\infty} c(v)~=~+\infty.
\]
\end{enumerate}
\subsection{System of first order Hamilton-Jacobi equations} 
 The   control system (\ref{eq:debt-GDP-ev})--(\ref{p}) is not standard. Indeed, the discount price $p$ in (\ref{p}) depends on the debt-to-income ratio not only at the present time $t$ but also at all future times. Here, we are mainly interested in construct optimal controls $(u^*,v^*)$ in feedback form:
\[
(u,v)~=~(u^*(x),v^*(x))\quad\mathrm{for}~ x\in [0,x^*].
\]

\begin{definition}[\textbf{Equilibrium solution in feedback form}]\label{def:eqsol} 
\emph{A couple of piecewise Lipschitz continuous  functions $(u^*(\cdot),v^*(\cdot))$ and l.s.c. 
$p^*(\cdot)$ provide an  equilibrium solution to the debt management problem (\ref{eq:debt-GDP-ev})-(\ref{J}), with 
continuous value function $V^*(\cdot)$, if
\begin{itemize}
\item[(i)] Given the price $p^*=p^*(x)$, one has that  $V^*$ is the value function and $(u^*(x),v^*(x))$ is the optimal feedback control, in connection with the deterministic control problem
\begin{equation}\label{eq:cost-feed}
\textrm{ minimize: }\int_0^{T_b} e^{-rt}[L(u(t))+c(v(t))]\, dt + e^{-r T_b} B,
\end{equation}
subject to 
\begin{equation}\label{eq:cont-feed}
\dot{x}(t)~=~\left(\dfrac{\lambda+r}{p^*(x)}-\lambda-\mu-v(t)\right)x(t)-\dfrac{u(t)}{p^*(x)},\hspace{1cm} x(0)~=~x_0,
\end{equation}
where the time $T_b$ is determined by (\ref{Tb}).
\item[(ii)] Given the feedback control $(u^*(x),v^*(x))$ in (\ref{eq:cont-feed}), for every $x_0\in [0,x^*]$ one has
\begin{multline}
p^*(x_0)~=~\int_0^{T_b}(r+\lambda)\exp\left\{-\int_0^t \bigl(\lambda +r+v^*(x(s)\bigr)\,ds\right\}\,dt+\\
+\exp\left\{-\int^{T_b}_{0}(r+\lambda+v^*(x(t)))\,dt\right\}\cdot \theta(x^*).
\end{multline}
\end{itemize}}
\end{definition}
Under the assumptions \textbf{(A1)-(A2)}, the Hamiltonian function $H:[0,x^*]\times \mathbb R\times [0,1]\to \mathbb R$ 
associated to the dynamics \eqref{eq:debt-GDP-ev} and to the cost functions $L,c$ is defined by
\begin{equation}\label{eq:detHamil}
H(x,\xi,p)~:=~\min_{u\in [0,1]}\left\{L(u)-u\,\dfrac{\xi}{p}\right\}+\min_{v\geq 0}\Big\{c(v)-vx\xi\Big\}
+ \left(\dfrac{\lambda+r}{p}-\lambda-\mu\right)x\,\xi.
\end{equation}
The Debt Management Problem  leads to the following implicit system of first order ODEs satisfied by the value
function $V$ and the  discounted rate $p$
\begin{equation}\label{eq:Sdode}
\begin{cases}
rV(x)~=~H(x,V'(x),p(x))\\[3mm]
(r+\lambda+v(x))p(x)-(r+\lambda)~=~H_{\xi}(x,V'(x),p(x))\cdot p'(x)\\[3mm]
v(x)~=~\displaystyle\underset{\omega\geq 0}{\mathrm{argmin}}\{c(\omega)-\omega xV'(x)\}
\end{cases}\end{equation}
with the boundary conditions
\begin{align}\label{eq:boundary}
V(0)~=~0,\qquad V(x^*)~=~B\qquad \textrm{and }\qquad p(0)~=~1,\qquad p(x^*)~=~\theta(x^*).
\end{align}
\subsection{Basic properties of $H$ and normal form of the system}
In this subsection, we  present some basic properties of the Hamiltonian function which will be  used to provide a semi-explicit formula for the optimal feed back strategy $(u^*,v^*)$. Let $L^\circ,c^\circ$ are  the convex conjugate of $L$ and $c$ (see in the Appendix for the notation).  We have that  
\begin{equation}\label{eq:detHamil-2}
-H(x,\xi,p)~\doteq~L^\circ\left(\dfrac{\xi}{p}\right)+c^\circ(x\xi)-\left(\dfrac{\lambda+r}{p}-\lambda-\mu\right)x\,\xi,
\end{equation}
and the map $\xi\mapsto -H(x,\xi,p)$ is convex and lower semicontinuous. Moreover, given $x>0$, $p\in ]0,1]$, $\xi\ge 0$, 
we denote by $u^*(\xi,p)\in [0,1]$ and $v^*(x,\xi)\in [0,+\infty[$ the unique elements of 
$\partial L^\circ\left(\dfrac{\xi}{p}\right)$ and $\partial c^\circ(x\xi)$, respectively, provided by Lemma \ref{lemma:subdiffcost}.
\begin{align*}
u^*(\xi,p)&~\doteq~\underset{u\in[0,1]}{\mathrm{argmin}}\left\{L(u) - u\dfrac{\xi}{p}\right\}~=~\begin{cases}0&\textrm{ if }0\le \xi< pL'(0)\cr
 (L')^{-1}(\xi/p)&\textrm{ if }\xi\ge pL'(0)>0
 \end{cases}
\end{align*}
and
\begin{align*}
v^*(x,\xi)&~\doteq~\underset{v\geq 0}{\mathrm{argmin}}\Big\{c(v)-vx\xi\Big\}~=~\begin{cases}0&\textrm{ if }0\le x\xi< c'(0)\cr (c')^{-1}(x\xi)&\textrm{ if }x\xi\ge c'(0)>0.\end{cases}
\end{align*}
It is clear that 
\begin{itemize}
\item for every $p\in]0,1]$ the map $\xi\mapsto u^*(\xi,p)$ is strictly increasing in $[pL'(0),+\infty[$, and $u^*(\cdot, p)\equiv 0$ in $[0,pL'(0)]$; 
\item for every $\xi\ge 0$ the map $p\mapsto u^*(\xi,p)$ is strictly decreasing in $[\xi/L'(0),1]$, and $u^*(\xi, \cdot)\equiv 0$ in $[0,\xi/L'(0)]$; 
\item for every $\xi>0$ the map $x\mapsto v^*(x,\xi)$ is strictly increasing in $[c'(0)/\xi,+\infty[$, and $v^*(\cdot,\xi)\equiv 0$ in $[0,c'(0)/\xi]$;
\item for every $x>0$ the map $\xi\mapsto v^*(x,\xi)$ is strictly increasing in $[c'(0)/x,+\infty[$, and $v^*(x,\cdot)\equiv 0$ in $[0,c'(0)/x]$.
\end{itemize}
From  Lemma \ref{lemma:Hamgrad}, the gradient of the Hamiltonian function $H(\cdot)$ can be expressed in terms of 
$u^*(\xi,p)$ and $v^*(x,\xi)$ at any point $(x,\xi,p)\in [0,+\infty[\times[0,+\infty[\times]0,1]$ by 
\begin{equation}\label{eq:gradH}\begin{cases}
\displaystyle H_{x}(x,\xi,p)&=~ \displaystyle \Big[(\lambda+r)-p(\lambda +\mu +v^*(x,\xi))\Big]\cdot \frac{\xi}{p}\\
\displaystyle H_{\xi}(x,\xi,p)&=~ \displaystyle \frac{1}{p}\cdot \Big[x\big((\lambda+r)-p(\lambda +\mu +v^*(x,\xi))\big)-u^*(\xi,p)\Big]\\
\displaystyle H_{p}(x,\xi,p)&=~ \displaystyle (u^*(\xi,p)-x(\lambda +r))\cdot \frac{\xi}{p^2}.
\end{cases}
\end{equation}
The following Lemma will catch some relevant properties of $H(\cdot)$ needed to study the system \eqref{eq:Sdode}.

\begin{lemma}\label{lemma:nonincmax}
Let $x\geq 0$ and $0<p\le 1$ be fixed, and set
\[H^{\max}(x,p)~\doteq~ \max_{\xi\geq 0}H(x,\xi,p).\] 
Then 
\begin{enumerate}
\item there exists $\xi^\sharp(x,p)>0$ such that, given $\eta>0$, the equation $r\eta=H(x,\xi,p)$ admits
\begin{itemize}
\item no solutions $\xi\in[0,+\infty)$ if $r\eta> H^{\max}(x,p)$,
\item $\xi^{\sharp}(x,p)$ as unique solution  if $r\eta=H^{\max}(x,p)$, 
\item exactly two distinct solutions $\{F^{-}(x,\eta,p),F^{+}(x,\eta,p)\}$ 
with \[0~<~F^{-}(x,\eta,p)~<~\xi^{\sharp}(x,p)~<~F^{+}(x,\eta,p)\] if $0<r\eta<H^{\max}(x,p)$,
\end{itemize}
\item we extend the definition of $\eta\mapsto F^{\pm}(x,\eta,p)$ by setting 
\[F^{\pm}\left(x,\dfrac{1}{r}\,H^{\max}(x,p),p\right)~=~\xi^{\sharp}(x,p),\]
thus for fixed $x>0$ , $p\in ]0,1]$, the maps $\eta\mapsto F^{-}(x,\eta,p)$ and $\eta\mapsto F^{+}(x,\eta,p)$
are respectively strictly increasing and strictly decreasing in $\left[0,\displaystyle\dfrac{H^{\max}(x,p)}{r}\right]$.
\item for all $0<\eta<H^{\max}(x,p)/r$ with $x>0$ and $p\in]0,1]$, we have 
\[\dfrac{\partial}{\partial \eta} F^{\pm}(x,\eta,p)~=~\dfrac{r}{H_\xi(x,F^{\pm}(x,\eta,p),p)},\]
\item The map $p\mapsto H^{\max}(x,p)$ is strictly decreasing on $]0,1]$ for every fixed $x\in ]0,x^*[$.
\end{enumerate}
\end{lemma}
\begin{proof}
Since for all fixed $x>0$, $0<p\le 1$ we have that $\xi\mapsto H(x,\xi,p)$ is the minimum of a family of affine functions of $\xi$, 
we have that the map $\xi\mapsto H(x,\xi,p)$ is concave down. 
Recalling \eqref{eq:gradH}, and the monotonicity properties of $u^*(\cdot,p)$ and $v^*(x,\cdot)$, since
\begin{itemize}
\item $H_{\xi}(x,\xi,p)=H_{\xi}(x,0,p)>0,\textrm{ for all }\xi\in [0,\min\{pL'(0),c'(0)/x\}]$,
\item $\xi\mapsto H_{\xi}(x,\xi,p),\textrm{ is strictly decreasing for all }\xi>\min\{pL'(0),c'(0)/x\}$,
\item $\displaystyle\lim_{\xi\to+\infty}H_{\xi}(x,\xi,p)=-\infty$,
\end{itemize}
we have that $\xi\mapsto H_{\xi}(x,\xi,p)$ vanishes in at most one point in $[0,+\infty)$, 
so $\xi\mapsto H(x,\xi,p)$ reaches its maximum value $H^{\max}(x,p)$ on $[0,+\infty)$ at a unique point $\xi^{\sharp}(x,p)$,
moreover it is strictly increasing for $0<\xi<\xi^{\sharp}(x,p)$ and strictly decreasing for $\xi>\xi^{\sharp}(x,p)$, with 
$\xi^{\sharp}(x,p)\ge \min\{pL'(0),c'(0)/x\}$. We define 
\begin{itemize}
\item the strictly increasing map $\eta\mapsto F^-(x,\eta,p)$, for $0<\eta<\displaystyle {1\over r}\cdot H^{\max}(x,p)$, to be the inverse of $\xi\mapsto \displaystyle \frac 1r \cdot H(x,\xi,p)$ for $0<\xi<\xi^{\sharp}(x,p)$;
\item the strictly decreasing map $\eta\mapsto F^+(x,\eta,p)$, for $0<\eta<\displaystyle {1\over r}\cdot H^{\max}(x,p)$, to be the inverse of $\xi\mapsto \displaystyle \frac 1r \cdot H(x,\xi,p)$ for $\xi>\xi^{\sharp}(x,p)$.
\end{itemize}
This proves (1) and (2). Now, set
\begin{align*}
u^{\sharp}(x,p)\doteq u^*(\xi^{\sharp}(x,p),p),&&v^{\sharp}(x,p)\doteq v^*(x,\xi^{\sharp}(x,p)).
\end{align*}
From \eqref{eq:gradH}, it holds
\begin{align}
\label{eq:Rmax}u^{\sharp}(x,p)&~=~\big[(\lambda+r)-(\lambda+\mu+v^{\sharp}(x,p))p\big]\cdot x,\\
\label{eq:Hmax}H^{\max}(x,p)&~=~L(u^{\sharp}(x,p))+c(v^{\sharp}(x,p)).
\end{align}
Moreover, 
\begin{itemize}
\item if $\xi^{\sharp}(x,p)\ge pL'(0)$ then
\begin{equation}
\label{eq:Xisharp}
\xi^{\sharp}(x,p)~=~pL'(u^{\sharp}(x,p))~=~pL'\left(\big[(\lambda+r)-(\lambda+\mu+v^{\sharp}(x,p))p\big]\cdot x\right),
\end{equation}
\item if $\xi^{\sharp}(x,p)\ge \displaystyle \dfrac{c'(0)}{x}$ then
\begin{equation}
\xi^{\sharp}(x,p)~=~\dfrac{c'(v^{\sharp}(x,p))}{x}.
\end{equation}
\end{itemize}
Conversely, for any fixed $x\geq 0$ and $0<p\le 1$, if 
\[u^*(\xi,p)~=~x\big((\lambda+r)-p(\lambda +\mu +v^*(x,\xi))\big),\]
then $\xi=\xi^{\sharp}(x,p)$, $v^*(x,\xi)=v^\sharp(x,p)$ and $u^*(\xi,p)=u^\sharp(x,p)$.
Indeed, this follows from the fact that $H_\xi(x,\xi,p)=0$ iff $\xi=\xi^\sharp(x,p)$.
\par\medskip\par

Consider the equation $\eta=H(x,\xi,p)/r$ for a given $\eta>0$,
and, noticing that, given $0<\xi<\xi^{\sharp}(x,p)$ we have
\begin{align*}
F^{-}(x,\eta,p)~=~F^{-}\left(x,\dfrac{1}{r}H(x,\xi,p),p\right)~=~\xi\qquad\forall 0<\xi<\xi^{\sharp}(x,p),\\
F^{+}(x,\eta,p)~=~F^{+}\left(x,\dfrac{1}{r}H(x,\xi,p),p\right)~=~\xi\qquad\forall \xi^{\sharp}(x,p)>\xi,
\end{align*}
and so (3) follows from the Inverse Function Theorem. To prove item (4), we notice that
\[\dfrac{d}{dp} H^{\max}(x,p)~=~\dfrac{d}{dp}H(x,\xi^\sharp(x,p),p)~=~H_p(x,\xi^\sharp(x,p),p).\]
Recalling \eqref{eq:gradH}, we have
\begin{eqnarray*}
H_p(x,\xi^{\sharp}(x,p),p)&=&\left[{u}^{\sharp}(x,p)-(r+\lambda)x\right]\cdot \dfrac{\xi^{\sharp}(x,p)}{p}\\
&=&-(\lambda+\mu+v^{\sharp}(x,p))x\xi^\sharp(x,p)~<~0\,,
\end{eqnarray*}
since for $x,p\ne 0$ we have $\xi^\sharp(x,p)>0$.
\end{proof}
\medskip

\begin{definition}[Normal form of the system]{\it 
Given $(x,p)\in ]0,x^*]\times ]0,1]$ such that $0<r\eta\le H^{\max}(x,p)$ we define the maps
\begin{equation}\label{eq:pnormal}
G^{\pm}(x,\eta,p)~=~\dfrac{(r+\lambda+v^*(x,F^{\pm}(x,\eta,p)))p-(r+\lambda)}{H_{\xi}(x,F^{\pm}(x,\eta,p),p)}.
\end{equation}
Notice that if $rV(x)>H^{\max}(x,p)$, then the first equation of \eqref{eq:Sdode} has no solution.
Otherwise, if $0<rV(x)<H^{\max}(x,p)$ this equation splits into 
\begin{align*}
\begin{cases}V'(x)~=~F^{-}(x,V(x),p(x)),\cr\cr p'(x)~=~G^{-}(x,V(x),p(x)),\end{cases}&&\textrm{ or }&&
\begin{cases}V'(x)~=~F^{+}(x,V(x),p(x)),\cr\cr p'(x)~=~G^{+}(x,V(x),p(x)).\end{cases}
\end{align*}
}
\end{definition}
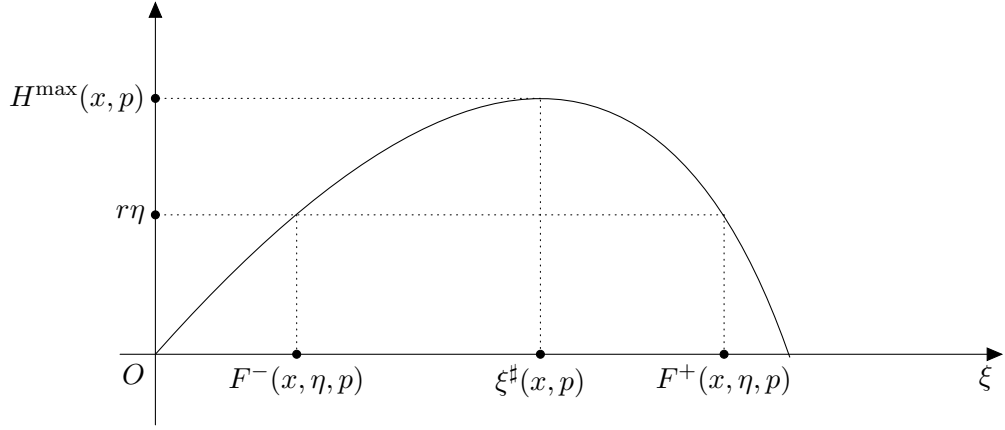
\begin{figure}[ht]
\centering
\begin{tikzpicture}[line cap=round,line join=round,>=triangle 45,x=0.3 \textwidth,y=0.3 \textwidth]
\draw[->,color=black] (-0.10,0) -- (2.40,0);\draw[->,color=black] (0,-0.20) -- (0,1);
\clip(-0.7,-0.20) rectangle (2.40,1.2);
\draw [samples=50,rotate around={-15:(0,0)},domain=0:1.737)] plot (\x,{1-(\x-1)^2});
\draw [dotted] (0.40,0.395)-- (0.40,0);\draw [dotted] (1.09,0.72)-- (1.09,0);
\draw [dotted] (1.61,0.395)-- (1.61,0);\draw [dotted] (1.61,0.395)-- (0,0.395);
\draw [dotted] (1.09,0.725)-- (0,0.725);
\node [below, color=black] at (0.40,0) {$F^{-}(x,\eta,p)$};
\node [below, color=black] at (1.09,0) {$\xi^\sharp(x,p)$};

\node [below, color=black] at (1.61,0) {$F^{+}(x,\eta,p)$};
\node [anchor=north east, color=black] at (0,0) {$O$};
\draw [color=black, fill=black] (0.40,0) circle (1.5pt);
\draw [color=black, fill=black] (1.09,0) circle (1.5pt);
\draw [color=black, fill=black] (1.61,0) circle (1.5pt);
\node [left, color=black] at (0,0.395) {$r\eta$};
\node [left, color=black] at (0,0.725)  {$H^{\max}(x,p)$};
\draw [color=black, fill=black] (0,0.395) circle (1.5pt);
\draw [color=black, fill=black] (0,0.725) circle (1.5pt);
\node [anchor=north east, color=black] at (2.40,0) {$\xi$};
\end{tikzpicture}
\label{fig:g53}
\caption{For $x\ge 0$, $p\in]0,1]$, the function $\xi\mapsto H(x,\xi,p)$ has a unique global maximum
$H^{\max}(x,p)$ attained at $\xi=\xi^\sharp(x,p)$. For $0<r\eta\leq H^{\max}$, the values $F^-(x,\eta, p)\leq \xi^\sharp(x,p)\leq F^+(x,\eta,p)$ are well defined.
Moreover, $F^{\pm}(x,\frac1rH^{\max}(x,p),p)=\xi^\sharp(x,p)$.}
\end{figure}
\begin{remark} 
Recalling \eqref{eq:debt-GDP-ev} and \eqref{eq:Rmax}, we observe that 
\begin{itemize}
\item The value $V'(x)=F^+(x,V(x),p)\geq \xi^\sharp(x,p)$ corresponds to the choice of an optimal control such that $\dot x(t)<0$. The total debt-to-ratio is decreasing.
\medskip
\item The value $V'(x)=F^-(x,V(x),p)\leq \xi^\sharp(x,p)$ corresponds to the choice of an optimal control such that $\dot x(x)>0$. The total debt-to-ratio is increasing.
\medskip
\item When $rV(x)= H^{\max}(x,p)$, then the value 
\[V'(x)~=~F^+(x,V(x),p)~=~F^-(x,V'(x),p)~=~\xi^\sharp(x,p)\] 
corresponds to the unique control strategy such that $\dot x(t)=0$. 
\end{itemize}
\end{remark}

\begin{remark}
We notice that if $0\le x\xi<\min\{xpL'(0),c'(0)\}$, since $u^*=v^*=0$, we have
\[\xi~=~F^{-}(x,\eta,p)~=~\dfrac{pr\eta}{(\lambda+r-p(\lambda+\mu))x},\]
in particular, if $0\le x\xi<\min\{xpL'(0),c'(0)\}$ we have that $\eta\mapsto F^{-}(x,\eta,p)$ is Lipschitz continuous, uniformly for $(x,p)\in [x_1,x^*]\times [p_1,1]$, 
for all $x_1\in]0,x^*]$, $p_1\in ]0,1]$.
If $x\xi>\min\{xpL'(0),c'(0)\}$, we have instead
\[H_{\xi\xi}(x,\xi,p)~\le~-\dfrac{1}{p}\cdot\min\left\{\dfrac{1}{p L''(u^*(\xi,p))}, \dfrac{x^2p}{c''(v^*(x,\xi))}\right\}.\]
\end{remark}
\begin{lemma}\label{lemma:holder}
Given $x_1\in]0,x^*]$, $p_1\in ]0,1]$, there exists a constant $C=C(x_1,p_1)$ such that
\[|F^{-}(x,\eta_1,p)-F^{-}(x,\eta_2,p)|~\le~ C\cdot |\eta_1-\eta_2|^{1/2},\]
for all $x\in [x_1,x^*]$, $p\in [p_1,1]$, $0<\eta_1,\eta_2\le \displaystyle \frac {1}{r} H^{\max}(x,p)$.
\end{lemma}
\begin{proof}
We distinguish two cases:
\begin{enumerate}
\item if $0\le x\xi<\min\{xpL'(0),c'(0)\}$, since $u^*=v^*=0$, we have
\[\xi~=~F^{-}(x,\eta,p)~=~\dfrac{pr\eta}{(\lambda+r-p(\lambda+\mu))x},\]
and so
\begin{align*}
|F^{-}(x,\eta_1,p)-F^{-}(x,\eta_2,p)|&~\le~\dfrac{pr}{(\lambda+r-p(\lambda+\mu))x} |\eta_1-\eta_2|\\
&~\le~\dfrac{\sqrt{2B}r}{(r-\mu)x_1} |\eta_1-\eta_2|^{1/2}.
\end{align*}
for all $x\in [x_1,x^*]$, $p\in [p_1,1]$, $0<\eta_1,\eta_2\le \displaystyle \frac {1}{r} H^{\max}(x,p)$.\par\medskip\par
\item If $x\xi>\min\{xpL'(0),c'(0)\}$, we have instead
\[H_{\xi\xi}(x,\xi,p)~\le~  -\dfrac{1}{p}\min\left\{\dfrac{1}{p L''(u^*(x,\xi,p))}, \dfrac{x^2p}{c''(v^*(x,\xi))}\right\},\]
thus, recalling that by assumption we have $L''(u)\ge \delta_0$ and $c''(v)\ge \delta_0$ for $0<u<1$ and $v\ge 0$,
we obtain
\[-H_{\xi\xi}(x,\xi,p)~\ge~ \dfrac{\min\{1,x_1^2p_1\}}{\delta_0}.\]
By applying Lemma \ref{lemma:holder-gen} to $f(\cdot)=-\displaystyle \frac{1}{r}H(x,\cdot,p)$, we have
\[|F^{-}(x,\eta_1,p)-F^{-}(x,\eta_2,p)|~\le~ \sqrt{\dfrac{2r\delta_0}{\min\{1,x_1^2p_1\}}}|\eta_2-\eta_1|^{1/2}.\]
\end{enumerate}
The proof is complete by choosing $C(x_1,p_1)\doteq \sqrt{\dfrac{2r\delta_0}{\min\{1,x_1^2p_1\}}}+\dfrac{\sqrt{2B}r}{(r-\mu)x_1}$.
\end{proof}
\section{An equilibrium solution to the Debt Management Problem}
In this section, we will provide a detail analysis on  the existence of a solution to the system of Hamilton-Jacobi equation \eqref{eq:Sdode} with boundary conditions \eqref{eq:boundary}
which yields an equilibrium solution to the Debt Management Problem  (\ref{eq:debt-GDP-ev})-(\ref{J}). A solution to will be constructed in the next following subsections.
\subsection{Constant strategies} 
We begin our analysis from the control strategies keeping the DTI constant in time, i.e., such that
the corresponding solution $x(\cdot)$ of \eqref{eq:debt-GDP-ev} is constant. In this case, there is no bankruptcy risk, i.e., $T_b=+\infty$.
\medskip

\begin{definition}[Constant strategies]\label{def:const-strat}{\it 
Let $\bar x>0$ be given.
We say that a pair $(\bar u,\bar v)\in [0,1[\times [0,+\infty[$ is a constant strategy for $\bar x$ if
\[\begin{cases}
\left[\left(\dfrac{\lambda+r}{\bar p}-\lambda-\mu -\bar v\right)\bar x - \dfrac{\bar u}{\bar p}\right]~=~0,\\ 
\bar p~=~\dfrac{r+\lambda}{r+\lambda + \bar v},
\end{cases}\]
where the second relation comes from taking $T_b=+\infty$ in \eqref{Tb}.}
\end{definition}
\par\medskip\par
From these equations, if a couple $(\bar u,\bar v)\in [0,1[\times [0,+\infty[$ is a constant strategy then it holds $(r+\lambda)(r-\mu)\bar x=(r+\lambda+\bar v)\bar u$. 
In this case, the borrower will never go bankrupt and thus the cost of this strategy in \eqref{J} is computed by 
\begin{align*}
\dfrac{1}{r}\cdot \Big[L(\bar u)+c(\bar v)\Big]&~=~\dfrac{1}{r}\cdot \left[L\left(\dfrac{(r+\lambda)(r-\mu)\bar x}{r+\lambda+\bar v}\right) + c\left(\bar v\right)\right]\\
&~=~\dfrac{1}{r}\cdot \left[L\left((r-\mu)\bar x\cdot\bar p\right) + c\left(\left(1-\dfrac{1}{\bar p}\right)(r+\lambda)\right)\right].\end{align*}
Notice that if $\bar x(r-\mu)>1$, since $0\le \bar u<1$ we must have $\bar v>1$ and $\bar p<1$, 
in particular if DTI is sufficiently large, every constant strategy needs to implement currency devaluation. 
A more precise estimate will be provided in Proposition \ref{prop:nondev}.

\medskip
We are now interested in the minimum cost of a strategy keeping the debt constant. To this aim, we first characterize the cost of a constant strategy in terms of the variables $x,p$.
\medskip

\begin{lemma}\label{lemma:cconvH} 
Given any $(x,p)\in ]0,+\infty[\times ]0,1]$, we have
\begin{equation}\label{eq:convH}H^{\max}(x,p)~=~\min\Big\{L(u)+c(v):\, u\in [0,1],\,v\ge 0,\,u=\big[(\lambda+r)-(\lambda+\mu+v)p\big]\cdot x\Big\}.
\end{equation}
Moreover, $(\hat u,\hat v)$ realizes the minimum in the right hand side of \eqref{eq:convH} if and only if
\[
\begin{cases}
\displaystyle c(\hat v)+px\hat v\xi^\sharp(x,p)&~=~\displaystyle\min_{\zeta\ge 0}\left\{px\xi^\sharp(x,p)\zeta+c(\zeta)\right\},\\ \\
\displaystyle L(\hat u)+\hat u\xi^\sharp(x,p)&~=~\displaystyle\min_{u\in[0,1]}\left\{\xi^{\sharp}(x,p)u+L(u)\right\}.
\end{cases}
\]
\end{lemma}
\begin{proof} 
Set $F(v):=f(v)+g(\Lambda v)$ where $f(\zeta)=c(\zeta)$ for $\zeta\ge 0$ and $f(\zeta)=+\infty$ if $\zeta<0$, $C(x,p)=\big[(\lambda+r)-(\lambda+\mu)p\big]\cdot x$,
$g(\zeta)=L(C(x,p)+\zeta)$ if $C(x,p)+\zeta\in [0,1]$
and $g(\zeta)=+\infty$ if $C(x,p)+\zeta\notin [0,1]$, and $\Lambda=-xp$.
By standard argument in convex analysis (see e.g. Theorem 4.2 and Remark 4.2 p. 60 of \cite{ET}), denoted by $f^\circ$, $g^\circ$ the convex conjugates of $f,g$ respectively, we have 
\begin{align*}
\inf_{v\in\mathbb R} F(v)&~=~\sup_{\nu\in\mathbb R}~\left[-f^\circ(\Lambda \nu)-g^\circ(-\nu)\right]\\
&~=~\sup_{\nu\in\mathbb R}~\left[\min_{\zeta\ge 0}\Big\{c(\zeta)+xp\nu\zeta\Big\}+\min_{C(x,p)+\zeta\in [0,1]}\Big\{L(C+\zeta)+\nu\zeta\Big\}\right]\\
&~=~\sup_{\nu\in\mathbb R}~\left[\min_{\zeta\ge 0}\Big\{c(\zeta)+xp\nu\zeta\Big\}+\min_{u\in [0,1]}\Big\{L(u)+\nu u\Big\}-C\nu\right]\\
&~=~\sup_{\xi\in\mathbb R}~\left[\min_{\zeta\ge 0}\Big\{c(\zeta)-x\xi\zeta\Big\}+\min_{u\in [0,1]}\Big\{L(u)-u\cdot\dfrac{\xi}{p}\Big\}+\dfrac{C(x,p)}{p}\cdot\xi\right]\\
&~=~\sup_{\xi\in\mathbb R}~H(x,\xi,p)=H^{\max}(x,p).
\end{align*}
Moreover, since $\displaystyle\sup_{\xi\in\mathbb R}H(x,\xi,p)$ is attained only at $\xi=\xi^\sharp(x,p)$ according to the strict concavity of $\xi\mapsto H(x,\xi,p)$, $(\hat u,\hat v)$
realizes the minimum in the right hand side of \eqref{eq:convH} if and only if
\[
\begin{cases}
f(\hat v)+f^\circ(\Lambda\xi^{\sharp}(x,p))-\Lambda\hat v\xi^\sharp(x,p)~=~0,\cr\cr
g(\Lambda \hat v)+g^\circ(-\xi^{\sharp}(x,p))+\Lambda \hat v \xi^\sharp(x,p)~=~0,
\end{cases}
\]
which implies $\hat v\ge 0$, $C(x,p)-px\hat v\in [0,1]$, and
\[
\begin{cases}
\displaystyle c(\hat v)+px\hat v\xi^\sharp(x,p)~=~\min_{\zeta\ge 0}~\{px\xi^\sharp(x,p)\zeta+c(\zeta)\},\cr\cr
\displaystyle L(C(x,p)-px\hat v)-px\hat v \xi^\sharp(x,p)~=~\min_{\nu\in\mathbb R}~\left\{\xi^{\sharp}(x,p)\nu+L(C(x,p)+\nu)\right\}.
\end{cases}
\]
The second relation can be rewritten as 
\[\displaystyle L(\hat u)+\hat u\xi^\sharp(x,p)~=~\min_{u\in[0,1]}\left\{\xi^{\sharp}(x,p)u+L(u)\right\},
\]
and this complete the proof.
\end{proof}
Formula \eqref{eq:convH} allows us to give a simpler characterization of the minimum cost of a strategy keeping the debt-to-income ratio constant in time.
Indeed, given $x\in [0,x^*]$, we select $(u(x),v(x))$ keeping the debt-to-income ratio constant in time. This defines uniquely a value $p=p(x)$ by Definition 
\ref{def:const-strat} and impose a relation between $u(x)$ and $v(x)$. 
Then we take the minimum over all the costs of such strategies, i.e., the right hand side of formula \eqref{eq:convH}.
This naturally leads to the following definition.
\medskip

\begin{definition}[Optimal cost for constant strategies]\label{def:optconst}\emph{
Given $x\in [0,x^*]$, we define
\[W(x)~=~\dfrac{1}{r}\cdot H^{\max}\left(x,p_c(x)\right)\]
where
\begin{equation}\label{eq:pcvc}
\begin{cases}p_c(x)~=~\dfrac{r+\lambda}{r+\lambda + v_c(x)},\\ \\
\displaystyle v_c(x)~=~\underset{v\ge 0}{\mathrm{argmin}}\left[L\left(\dfrac{(r+\lambda)(r-\mu)x}{r+\lambda+v}\right)+c(v)\right].\end{cases}
\end{equation}
For every $x\in [0,x^*]$, $W(x)$ denotes the minimum cost of a strategy keeping the DTI ratio constant in time.}
\end{definition}

The next results proves that if the debt-to-income ratio is sufficiently small, the optimal strategy keeping it constant does not use the devaluation of currency.
\medskip

\begin{proposition}[Non-devaluating regime for optimal constant strategies]\label{prop:nondev}
Let $x_c\ge 0$ be the unique solution of the following equation in $x$ 
\[(r+\lambda)c'(0)~=~(r-\mu)xL'\left((r-\mu)x\right).\]
Then
\begin{itemize}
\item for all $x\in [0,\min\{x_c,x^*\}]$ we have $W(x)=\dfrac{1}{r}\cdot L((r-\mu)x)$ and $p_c(x)=1$,
\item for all $x\in ]\min\{x_c,x^*\},x^*]$ we have 
\begin{align*}
W(x)&~=~\dfrac{1}{r}\left[L\left(\dfrac{(r+\lambda)(r-\mu)x}{r+\lambda+v_c(x)}\right)+c(v_c(x))\right],\\
p_c(x)&~=~\dfrac{r+\lambda}{r+\lambda + v_c(x)}<1,
\end{align*}
where $v_c(x)>0$ solves the following equation in $v$ 
\[c'(v)~=~\dfrac{(r+\lambda)(r-\mu)x}{(r+\lambda+v)^2}\cdot L'\left(\dfrac{(r+\lambda)(r-\mu)x}{r+\lambda+v}\right).\]
\item for every $x\in]0,x^*[$ we have
\begin{equation}\label{eq:k-slop}
W'(x)~=~\dfrac{r-\mu}{r} p_c(x) L'(p_c(x)(r-\mu)x)<\xi^\sharp(x,p_c(x)).
\end{equation}
\end{itemize}
\end{proposition}
\begin{proof}
Given $x\in]0,x^*[$, we define the convex function 
\[F^x(v)~\doteq~\begin{cases}\dfrac{1}{r}\cdot \left[L\left(\dfrac{(r+\lambda)(r-\mu)x}{r+\lambda+v}\right)+c(v)\right],&\textrm{ if }v\ge 0,\,\dfrac{(r+\lambda)(r-\mu)x}{r+\lambda+v}\in [0,1],\cr\cr
+\infty,&\textrm{otherwise}.\end{cases}\]
We compute
\[
\dfrac{d}{dv}F^x(v)~=~\dfrac{1}{r}\cdot \left[c'(v)-L'\left(\dfrac{(r+\lambda)(r-\mu)x}{r+\lambda+v}\right)\dfrac{(r+\lambda)(r-\mu)x}{(r+\lambda+v)^2}\right],
\]
which is monotone increasing and satisfies $\displaystyle\lim_{v\to+\infty}\dfrac{d}{dv}F^x(v)=+\infty$,
\[
\dfrac{d}{dv}F^x(v)~\ge~\dfrac{d}{dv}F^x(0)~=~\dfrac{1}{r}\cdot \left[c'(0)-L'\left((r-\mu)x\right)\dfrac{(r-\mu)x}{r+\lambda}\right].
\]
Two cases may occur:
\begin{itemize}
\item If $\dfrac{d}{dv}F^x(0)\ge 0$, we have that $v=0$ realizes the minimum of $F$ on $[0,+\infty[$. This occours when $x\in [0,\min\{x_c,x^*\}]$ where 
$x_c$ is the unique solution of 
\[(r+\lambda)c'(0)~=~(r-\mu)xL'\left(\dfrac{(r+\lambda)(r-\mu)x}{r+\lambda}\right),\]
and it implies $W(x)=\dfrac{1}{r}\cdot L((r-\mu)x)$ and $p_c(x)=1$.\par\medskip\par
\item If we have $\min\{x_c,x^*\}<x\le x^*$, then there exists a unique point $v_c(x)>0$ such that $F'(v_c(x))=0$, and this point is characterized by 
\[
c'(v_c(x))~=~\dfrac{(r+\lambda)(r-\mu)x}{(r+\lambda+v_c(x))^2}\cdot L'\left(\dfrac{(r+\lambda)(r-\mu)x}{r+\lambda+v_c(x)}\right).
\]
The remaining statements follows noticing that for $\min\{x_c,x^*\}<x\le x^*$ we have
\begin{align*}
W'(x)&~=~\ \dfrac{\partial F^x}{\partial x}(v_c(x))+\dfrac{\partial}{\partial v}F^x(v_c(x)) \cdot v'_c(x)~=~\dfrac{\partial F^x}{\partial x}(v_c(x))\\
&~=~\ \dfrac{r-\mu}{r} p_c(x) L'(p_c(x)(r-\mu)x),
\end{align*}
and deriving the explicit expression of $W(x)$ for $[0,\min\{x_c,x^*\}]$ yields the same formula.
Notice that, by \eqref{eq:Xisharp}, we have
\begin{align*}
\xi^\sharp(x,p_c(x))&~=~p_c(x) L'\left(\big[(\lambda+r)-(\lambda+\mu+v^{\sharp}(x,p_c(x)))p_c(x)\big]\cdot x\right)\\
&~=~ p_c(x)L'\left(\left[(\lambda+r)-(\lambda+\mu+v^{\sharp}(x,p_c(x)))\cdot \dfrac{\lambda+r}{\lambda+r+v^{\sharp}(x,p_c(x))}\right]\cdot x\right)\\
&~=~ p_c(x) L'(p_c(x)(r-\mu)\cdot x)>W'(x),
\end{align*}
where we used the fact that $L'$ is strictly increasing and, since the argument of $L'$ must be nonnegative, we have
\[\dfrac{\lambda+r}{\lambda+\mu+v^{\sharp}(x,p_c(x))}~\ge~p_c(x),\]
\end{itemize}
and the proof is complete.
\end{proof}
\subsection{Existence of an equilibrium solution.} 
We are now ready to establish an existence result of a equilibrium solution to  the debt management problem  (\ref{eq:debt-GDP-ev}) - (\ref{J}). Before going to state our main theorem, we recall from Proposition \ref{prop:nondev} that $v_{c}$ is the unique solution to 
\[c'(v)~=~\dfrac{(r+\lambda)(r-\mu)x}{(r+\lambda+v)^2}\cdot L'\left(\dfrac{(r+\lambda)(r-\mu)x}{r+\lambda+v}\right),\]
and 
\[p_c(x^*)~=~\dfrac{r+\lambda}{r+\lambda + v_c(x^*)} ~<~1\,,\]
\begin{equation}\label{eq:w11}
W(x^*)~=~\dfrac{1}{r}\left[L\left(\dfrac{(r+\lambda)(r-\mu)x^*}{r+\lambda+v_c(x^*)}\right)+c(v_c(x^*))\right]\,.
\end{equation}

\begin{theorem}\label{thm:main2} 
Assume that the cost functions $L$ and $c$ satisfies the assumptions \textbf{(A1)-(A2)}, and moreover
\begin{equation}\label{eq:x-large}
W(x^*)~>~B\qquad\textrm{ and }\qquad \theta(x^*)~\le~ p_c(x^*).
\end{equation}
Then the debt management problem  (\ref{eq:debt-GDP-ev}) - (\ref{J}) admits an equilibrium solution $(u^*,v^*,p^*)$ associated to  Lipschitz continuous value functions $V^*$ in feedback form such that $p^*$ is decreasing, $V^*$ is strictly increasing and
\[
V^*(x)~\leq~W^*(x)\qquad\forall x\in [0,x^*].
\]
\end{theorem}
Toward the proof of this theorem, we first study  basic properties of the backward solutions of the  system of implicit ODEs \eqref{eq:Sdode}. In fact, an equilibrium solution will be constructed by a suitable concatenation of backward solutions.
\subsubsection{Backward solutions}  We first define the backward solution to the system \eqref{eq:Sdode} starting from $x^*$.
%
\begin{definition}[Backward solution for $x^*$]{\it 
Let $x\mapsto (Z(x,x^*),q(x,x^*))$ be the backward solution of the system of ODEs
\begin{equation}\label{eq:VB}
\begin{cases}Z'(x)&~=~F^{-}(x,Z(x),q(x)),\\[4mm] q'(x)&~=~G^{-}(x,Z(x),q(x)), \end{cases}
\qquad\textrm{with}\qquad\begin{cases}Z(x^*)&~=~B\,,\\[4mm] q(x^*)&~=~\theta(x^*).\end{cases}
\end{equation}
with $H_\xi(x,F^{-}(x,Z(x),q(x)),q(x))\ne 0$.
}
\end{definition}
The following Lemma states some basic properties of the backward solution. In particular, the backward solution $Z(\cdot,x^*)$, starting from $B$ at $x^*$ with $W(x^*)<B$, 
survives backward at least  until the first intersection with the graph of $W(\cdot)$. Moreover, in this interval is monotone increasing and positive. In the same way, $q(\cdot,x^*)$ 
is always in $]0,1]$.
\medskip
\begin{proposition}\label{prop:pro1}[Basic properties of the backward solution]
Set
\[
x^*_{W}~:=~\begin{cases}0,&\textrm{ if }Z(x,x^*)<W(x)\textrm{ for all }x\in]0,x^*[,\cr\cr
\sup\{x\in ]0,x^*[:\, Z(x,x^*)\ge W(x)\},&\textrm{ otherwise }.\end{cases}
\]
Assume that 
\begin{equation}\label{eq:Asp1}
W(x^*)~>~B\qquad\textrm{ and }\qquad \theta(x^*)~<~\dfrac{r+\lambda}{r+\lambda+v^*(x^*,F^-(x^*,B,\theta(x^*))) }\,.
\end{equation}
Denote by $I_{x^*}\subseteq [0,x^*]$ the maximal domain of the backward equation \eqref{eq:VB}, 
define $y(x)$ to be the maximal solution of
\[\begin{cases}
\dfrac{dy}{dx}(x)&~=~\dfrac{1}{H_{\xi}\left(x, Z'(x,x^*),q(x,x^*)\right)},\\  y(x^*)&~=~0,   
\end{cases}\]
and let $J_{x^*}$ the intersection of its domain with $[0,x^*]$. 
Then
\begin{enumerate}
\item $I_{x^*}\supseteq J_{x^*}\supseteq ]x^*_{W},x^*[$;
\item $Z(\cdot,x^*)$ is strictly monotone increasing in $]x^*_{W},x^*[$, and $Z(x,x^*)>0$ for all $x\in ]x^*_{W},x^*]$;
\item $q(x,x^*)\in ]0,1]$ for all $x\in ]x^*_{W},x^*]$. 
\end{enumerate}
\end{proposition}
\begin{proof} \textbf{1.} We first claim that $q(\cdot,x^*)$ is non-increasing on $J_{x^*}\bigcap ]x^*_{W},x^*[$ and thus
\begin{equation}\label{eq:de-p}
q'(x,x^*)~=~\dfrac{[r+\lambda+v^*(x,Z'(x,x^*))]\cdot q(x,x^*)-(r+\lambda)}{H_{\xi}(x,Z'(x,x^*),q(x,x^*))}~\le~0\quad\textrm{ for all }x\in J_{x^*}\cap ]x^*_{W},x^*[\,.
\end{equation}
By contradiction, assume that there exists $x_1\in J_B\cap ]x_{BW},x^*[$  such that 
\begin{equation}\label{eq:as1}
q'(x_1,x^*)~=~\dfrac{[r+\lambda+v^*(x_1,Z'(x_1,x^*))]\cdot q(x_1,x^*)-(r+\lambda)}{H_{\xi}(x_1,Z'(x,x^*),q(x,x^*))}~=~0,\quad q''(x_1,x^*)~<~0\,.
\end{equation}
This yields 
\[r+\lambda~=~[r+\lambda+v^*(x_1,Z'(x_1,x^*))]\cdot q(x_1,x^*),\qquad q(x_1,x^*)~>~0.\]
Two cases are considered:
\begin{itemize}
\item if $x_1Z'(x_1,x^*)\leq c'(0)$ then, recalling the monotonicity of $Z'(\cdot,x^*)$, we have that \newline $xV'(x,x^*)\leq c'(0)$ for all $x\in J_{x^*}\cap ]x^*_{W},x^*[$ 
satisfying $x\le x_1$, and so
\[v^*(x,Z'(x,x^*))~=~0\quad \textrm{ for all }x\in J_{x^*}\cap ]x^*_{W},x^*[~\textrm{with}~x\le x_1.\]
Thus, $q(x_1,x^*)=1$ and 
\[
q'(x,x^*)~=~\dfrac{[r+\lambda]\cdot [q(x,x^*)-1]}{H_{\xi}(x,Z'(x,x^*),q(x,x^*))}\quad\forall x\in J_{x^*}\cap ]x^*_{W},x^*[~\textrm{with}~x\le x_1.
\]
This implies that  $q(x,x^*)=1$ for all $x\in J_{x^*}\cap ]x^*_{W},x^*[$ with $x\le x_1$. In particular, we have $q''(x_1,x^*)=0$, which yields a contradiction.
\item If $x_1Z'(x_1,x^*)>c'(0)$ then 
\[\dfrac{d}{dx}(v^*(x_1,Z'(x_1,x^*)))~=~\dfrac{Z''(x_1,x^*)x_1+Z'(x_1,x^*)}{c''(x_1Z'(x_1,x^*))}~>~0.\]
From the first equation of \eqref{eq:Sdode} and \eqref{eq:gradH}, it holds
\begin{align*}
rZ'(x_1,x^*)&~=~H_x(x_1,Z',q)+H_{\xi}(x_1,Z',q)\cdot Z''(x_1,x^*)+H_{p}(x_1,Z,q)\cdot q'(x_1,x^*)\\
&~=~\Big[(\lambda+r)-q(x_1,x^*)(\lambda +\mu +v^*(x,Z')\Big]\cdot \frac{Z'}{q}+H_{\xi}(x_1,Z',q)\cdot Z''(x_1,x^*)\\
&~=~(r-\mu)\cdot Z'(x_1,x^*)+H_{\xi}(x_1,Z',q)\cdot Z''(x_1,x^*).
\end{align*}
Observe that  $Z'(x_1,x^*)>0$ and $H_{\xi}(x_1,Z'(x_1,x^*),q(x_1,x^*))>0$,  one obtains that 
\[
Z''(x_1,x^*)~=~\dfrac{\mu Z'(x_1,x^*)}{H_{\xi}(x_1,Z'(x_1,x^*),q(x_1,x^*))}~>~0\,.
\]
Taking the derivative respect to $x$ in both sides of the second equation of \eqref{eq:Sdode}, we have 
\begin{multline*}
\left [r+\lambda+(v^*(x,Z'(x,x^*))\right]\cdot q'(x,x^*)+q(x,x^*)\cdot\dfrac{d}{dx}v^*(x,Z'(x,x^*))\\
~=~q''(x,x^*)H_{\xi}(x,Z'(x,x^*),q(x,x^*))+q'(x,x^*)\dfrac{d}{dx}H_{\xi}(x,Z'(x,x^*),q(x,x^*))\,.
\end{multline*}
Recalling \eqref{eq:as1}, we obtain that 
\begin{equation}\label{eq:p''}
q''(x_1,x^*)~=~\dfrac{q(x_1,x^*)}{H_{\xi}(x_1,Z',q)}\cdot \dfrac{d}{dx}v^*(x_1,Z'(x_1,x^*))~>~0.
\end{equation}
and it yields a contradiction.
\end{itemize}
\par\medskip\par

\noindent Assume that there exists $x_2\in J_{x^*}\cap ]x^*_{W},x^*[$ such that $H_{\xi}(x_2,Z'(x_2,x^*),q(x_2,x^*))=0$. Then  
\[
\xi^{\sharp}(x_2,q(x_2,x^*))~=~Z'(x_2,x^*), \qquad Z(x_2,x^*)~=~\dfrac{1}{r}\cdot H^{\max}(x_2,q(x_2,x^*)),
\]
and
\[
u^{\sharp}(x_2,q(x_2,x^*))~=~\big[(\lambda+r)-(\lambda+\mu+v^{\sharp}(x_2,q(x_2,x^*)))q(x_2,x^*)\big]\cdot x_2\,.
\]
Since $q(x_2,x^*)\leq \dfrac{r+\lambda}{r+\lambda+v^{\sharp}(x,Z'(x_2,x^*))}$, we estimate
\begin{align*}
H^{\max}&(x_2,q(x_2,x^*))~=~L(u^{\sharp}(x_2,q(x_2,x^*)))+c(v^{\sharp}(x_2,q(x_2,x^*)))\\
&~=~L\left(\big[(\lambda+r)-(\lambda+\mu+v^{\sharp}(x_2,q(x_2,x^*)))q(x_2,x^*)\big]\cdot x_2\right)+c(v^{\sharp}(x_2,q(x_2,x^*)))\\
&~\geq~L\left(\dfrac{r+\lambda(r-\mu)x_2}{\lambda+\mu+v^{\sharp}(x_2,q(x_2,x^*))}\right)+c(v^{\sharp}(x_2,q(x_2,x^*)))\\
&~\geq~H^{\max}(x_1,p_c(x_2)).
\end{align*}
Thus, 
\[
Z(x_2,x^*)~=~\dfrac{1}{r}\cdot H^{\max}(x_2,q(x_2,x^*))~\geq~\dfrac{1}{r}\cdot H^{\max}(x_2,p_c(x_2))~=~W(x_2),
\]
and this yields a contradiction.
\medskip

\textbf{2.} By construction, $y(\cdot)$ is strictly monotone and invertible in $]x^*_{W},x^*]$, let $x=x(y)$ be its inverse, from the inverse function theorem we get
\[\begin{cases}
\dfrac{d}{dy}~Z(x(y),x^*)&~=~Z'(x(y),x^*)\cdot H_{\xi}\left(x(y),Z'(x(y),x^*),q(x(y),x^*)\right),\\ \\
\dfrac{d}{dy}~q(x(y),x^*)&=~q'(x(y),x^*)\cdot H_{\xi}\left(x(y),Z'(x(y),x^*),q(x(y),x^*)\right).
\end{cases}\]
Since the map $\xi\mapsto H(x,\xi,q)$ is concave, it holds
\[H_{\xi}(x,0,q(x,x^*))~\ge~H_{\xi}(x,\xi,q(x,x^*))~\ge~H_{\xi}\left(x,Z'(x,x^*),q(x,x^*)\right),\]
for all $\xi\in \left[0,Z'(x,x^*)\right]$. Thus,
\begin{align*}
r Z(x(y),x^*)&~=~H\left(x(y), Z'(x(y),x^*),q(x(y),x^*)\right)~=~\int_0^{Z'(x(y),x^*)}H_{\xi}(x,\xi,q(x(y),x^*))\,d\xi\\
&~\ge~Z'(x(y),x^*)\cdot H_{\xi}(x,Z(x(y),x^*),q(x(y),x^*))~=~\dfrac{d}{dy}Z(x(y),x^*),
\end{align*}
and this implies that 
 $$
 Z(x,x^*)~\ge~ Be^{ry(x)}~>~0\qquad\forall x\in ]x^*_{W},x^*].
 $$
With a similar argument for $q(\cdot,x^*)$, we obtain
\[
\big[r+\lambda+v^*(x(y),Z'(x(y),x^*)\big]\cdot q(x(y),x^*)-(r+\lambda)~=~\dfrac{d}{dy}q(x(y),x^*)).
\]
Hence,
\[(r+\lambda)(q(x(y),x^*)-1)~\le~\dfrac{d}{dy}q(x(y),x^*)~\le~ \big[r+\lambda+v^*(x(y),Z'(x(y),x^*)\big]\cdot q(x(y),x^*),
\]
and  this yields 
\[
q(x,x^*)~\le ~1\qquad\mathrm{and}\qquad q(x,x^*)~\ge~\theta(x^*)\cdot e^{(r+\lambda+v^*(x,Z'(x,x^*))y(x)}~>~0
\]
for all $x\in I_{x^*}\cap[0,x^*]$. In particular,  $q(x,x^*)\in ]0,1]$ for all $x\in ]x^*_{W},x^*]$.
\end{proof}
\par\medskip\par
As far as the graph of $Z(\cdot,x^*)$ intersects the graph of $W(\cdot)$,  $Z(\cdot,x^*)$ is no longer optimal. The following lemma is to investigate  the local behavior of $Z(\cdot,x^*)$ and $W(\cdot)$ near to an intersection of their graphs.

\begin{lemma}[Comparison between optimal constant strategy and backward solution]\label{lemma:compW}
Let $I\subseteq ]0,x^*[$ be an open interval, $(Z,q):I\to [0,+\infty[\times]0,1[$ be a backward solution, and $\bar{x}\in \bar I$. 
If 
\[
\lim_{\substack{x\to \bar x\\ x\in I}}Z(x)~=~W(\bar x)
\]
then $p_c(\bar x)\ge \displaystyle\limsup_{\substack{I\ni x\to\bar x}} q(x)$ and $W'(x)< F^{-}(x,W(x),p_c(x))$.
\end{lemma}
\begin{proof}
Let $\{x_j\}_{j\in\mathbb N}\subseteq I$ be a sequence converging to $\bar x$ and $q_{\bar x}\in [0,1]$ be such that $\displaystyle q_{\bar x}=\limsup_{x\to\bar x^+} q(x)=\lim_{j\to \infty}q(x_j)$. We have 
\[H^{\max}(x,p_c(x))~=~\lim_{j\to +\infty} H\left(x_j,Z'(x_j),q(x_j)\right)~\le~ \lim_{j\to +\infty} H^{\max}(x_j,q(x_j))~=~H^{\max}(\bar x,q_{\bar x}).\]
From \ref{lemma:nonincmax} (4), it holds that  $p_c(\bar x)\ge q_{\bar x}$. By Proposition \ref{prop:nondev}, we have $W'(\bar x)<\xi^\sharp(\bar x,p_c(\bar x))$, and so
\[
H(\bar x,W'(\bar x),p_c(\bar x))~<~H^{\max}(\bar x,p_c(\bar x))~=~rW(\bar x).
\] 
Thus, by applying the strictly increasing map $F^{-}(\bar x,\cdot,p_c(x))$ on both sides, we obtain $W'(x)< F^{-}(x,W(x),p_c(x))$.
\end{proof}
\medskip

Since the functions $F^-(x,Z,q)$ and $G^{-}(x,Z,q)$ are smooth for $H_{\xi}(x,Z,q)\neq 0$ but not only H\"{o}lder continuous with respect to $Z$ near to the surface
\[
\Sigma~=~\left\{(x,Z,q)\in\R^3~:~H_{\xi}(x,Z,q)=0\right\}\,.
\]
Given any  $x_0\in [0,x^*)$,  the definition of the solution of  the Cauchy problem 
\begin{equation}\label{eq:V-x0}
\begin{cases}Z'(x)&=~F^{-}(x,Z(x),q(x)),\\[4mm] q'(x)&=~G^{-}(x,Z(x),q(x)), \end{cases}
\qquad\textrm{with}\qquad\begin{cases}Z(x_0)&=~W(x_0)\,,\\[4mm] q(x_0)&=~p_c(x_0).\end{cases}
\end{equation}
requires some care. For any $\varepsilon>0$, we denote by $Z_{\varepsilon}(\cdot,x_0),q_{\varepsilon}(\cdot,x_0)$ the backward solution to \eqref{eq:V-x0} with the terminal data
\[
Z_{\varepsilon}(x_0,x_0)~=~W(y_0)-\varepsilon\qquad\textrm{ and }\qquad q_{\varepsilon}(x_0,x_0)~=~p_c(x_0).
\]
With the same argument in the proof of Proposition \ref{prop:pro1}, the solution is uniquely defined on a maximal interval $[a_{\varepsilon}(x_0),x_0]$ such that $Z_{\varepsilon}(\cdot,x_0)$ is increasing, $q_{\varepsilon}(\cdot,x_0)$ is decreasing and 
\[
Z_{\varepsilon}(a_{\varepsilon}(x_0),x_0)~=~W(a_{\varepsilon}(x_0)),\qquad \qquad q_{\varepsilon}(a_{\varepsilon}(x_0),x_0)~\le~ p_{c}(a_{\varepsilon}(x_0)).
\]
Let $x^{\flat}$ be the unique solution to the equation
\begin{equation}\label{eq:x-flat}
c'(0)~=~x\cdot L'((r-\mu)x)\,.
\end{equation}
It is clear that $0<x^{\flat}<x_c$ where $x_c$ is defined in Proposition \ref{prop:nondev} as the unique solution to the equation
\[
(r+\lambda)c'(0)~=~(r-\mu)xL'\left((r-\mu)x\right).
\]
Two cases are considered:
\medskip

\underline{\textbf{CASE 1:}} For any $x_0\in ]0,x^{\flat}]$, we claim that  
\[
a_{\varepsilon}(x_0)~=~0,\qquad q_{\varepsilon}(x,x_0)~=~1\qquad\forall x\in [0,x_0]\,,
\]
and $Z_{\varepsilon}(\cdot,x_0)$ solves backward the following ODE  
\begin{equation}\label{eq:odep=1}
Z'(x)~=~F^-(x,Z(x),1),\qquad Z(x_0)~=~W(x_0)-\varepsilon
\end{equation}
for $\varepsilon>0$ sufficiently small. Indeed, let $Z_1$ be the unique backward solution of \eqref{eq:odep=1}. From \eqref{eq:Xisharp}, it holds
\[
F^-(x,W(x),1)~=~\xi^{\sharp}(x,1)~=~L'((r-\mu)x)~>~\dfrac{r-\mu}{r}\cdot L'((r-\mu)x)~=~W'(x)
\]
for all $x\in ]0,x^{\flat}]$. As in \cite{BMNP}, a contradiction argument yields
\[
0~<~Z_1(x)~<~W(x)\qquad\forall x\in ]0,x_0]\,.
\]
Thus, $Z_1$ is well-defined on $[0,x_0]$ and $Z_1(0)=0$. On the other hand, it holds
\[
Z'(x_1)~=~F^-(x,Z(x),1)~\le~ \xi^{\sharp}(x,1)~=~L'((r-\mu)x)~\le~ L'((r-\mu)x^{\flat})
\]
for all $x\leq x^{\flat}$ and \eqref{eq:x-flat} implies that 
\[
v^*(x,Z'_1(x))~=~0\qquad\forall x\in [0,x^{\flat}]\,.
\]
Therefore, $(Z_1(x),1)$ solves \eqref{eq:V-x0} and the uniqueness yields 
\[
Z_{\varepsilon}(x,x_0)~=~Z_1(x)\qquad\textrm{ and }\qquad q_{\varepsilon}(x,x_0)~=~1\qquad\forall x\in [0,x_0]\,.
\] 
Thanks to the monotone increasing property of the map $\xi\to F^-(x,\xi,1)$, a pair\newline $(Z(\cdot,x_0),q(\cdot,x_0))$ denoted by 
\[
q(x,x_0)~=~1\qquad\textrm{ and }\qquad Z(x,x_0)~=~\sup_{\varepsilon>0}Z_{\varepsilon}(x,x_0)\qquad\forall x\in [0,x_0]
\]
is the unique solution of \eqref{eq:V-x0}. If the initial size of the debt is $\bar{x}\in [0,x_0]$ we think of $Z(\bar{x},x_0)$ is as the expected cost of 
\eqref{eq:cost-feed}-\eqref{eq:cont-feed} with $p(\cdot,x_0)=1$, $x(0)=x_0$ achieved by the feedback strategies
\begin{equation}\label{eq:u-1xx}
u(x,x_0)~=~\underset{w\in [0,1]}{\mathrm{argmin}}\left\{ L(w) - \, Z'(x,x_0)\cdot w\right\},\qquad v(x,x_0)~=~0
\end{equation}
for all $x\in [0,x_0]$. With this strategy, the debt has the asymptotic behavior $x(t)\to x_0$ as $t\to\infty$.
\medskip

\underline{\textbf{CASE 2:}} For $x_0\in (x^{\flat},x_W^*]$,  system of ODEs \eqref{eq:V-x0} does not admit a unique solution in general since it is not monotone.
The following lemma will provide the existence result of \eqref{eq:V-x0} for all $x_0\in (x^{\flat},x^*_{W}]$.
\medskip

\begin{lemma}\label{lemma:bw} There exists a constant $\delta_{\flat}>0$ depending only on $x^{\flat}$ such that  for any $x_0\in \left(x^{\flat},x^*_W\right)$, it holds  
\[
x_0-a_{\varepsilon}(x_0)~\ge~ \delta_{x^\flat}\qquad\forall \varepsilon\in (0,\varepsilon_0)
\]
for some $\varepsilon_0>0$ sufficiently small.
\end{lemma}
\begin{proof} From \eqref{eq:k-slop} and \eqref{eq:Xisharp}, it holds
\[
\inf_{x\in [x^{\flat},x_W^*]}\left\{\xi^{\sharp}(x,p_c(x))-W'(x)\right\}~=~\delta_{1,\flat}~>~0.
\]
In particular, we have 
\[
F^{-}(x_0,W(x_0),p_{c}(x_0))-W'(x_0)~=~\delta_{1,\flat}.
\]
By continuity of the map $\eta\mapsto F^-(x_0,\eta,p_c(x_0))$ on $[0,W(x)]$, we can find a constant $\varepsilon_1>0$ sufficiently small such that 
\[
F^-(x_0,\eta,p_c(x_0))~\ge~ W'(x_0)+\dfrac{\delta_{1,\flat}}{2}\quad\forall \xi\in [W(x_0)-\varepsilon_1,W(x_0)].
\]
On the other hand, the continuity of $W'$ yields
\[
\delta_{2,\flat}~=~\sup\left\{s\geq 0~\Big|~W'(x_0-\tau)~<~W'(x_0)+\dfrac{\delta_{1,\flat}}{4}\quad\forall \tau \in [0,s]\right\}>0.
\] 
For a fixed $\varepsilon\in (0,\varepsilon_1)$, denote by 
\[
x_1~:=~\inf\left\{s\in (0,x_0]~\Big|~F^{-}\big(x,Z_{\varepsilon}(x,x_0),q_{\varepsilon}(x,x_0)\big)>W'(x)~\forall x\in (s,x_0]\right\}.
\]
If $x_1>x_0-\delta_{2,\bar{x}}$ then it holds
\begin{equation}\label{eq:cd1}
F^{-}\big(x_1,Z_{\varepsilon}(x_1,x_0),q_{\varepsilon}(x_1,x_0)\big)~=~W'(x_1)~\le~ W'(x_0)+\dfrac{\delta_{1,\flat}}{4}
\end{equation}
and there exists $x_2\in (x_1,x_0]$ such that 
\begin{equation}\label{eq:cd2}
F^{-}\big(x_2,Z_{\varepsilon}(x_2,x_0),q_{\varepsilon}(x_2,x_0)\big)~=~W'(x_0)+\dfrac{\delta_{1,\flat}}{2}
\end{equation}
and 
\begin{equation}\label{eq:k-cod}
F^{-}\left(x,Z_{\varepsilon}(x,x_0),q_{\varepsilon}(x,x_0)\right)~\le~ W'(x_0)+\dfrac{\delta_{1,\flat}}{2}\qquad\forall x\in [x_1,x_2].
\end{equation}
Recalling that $(x,\eta,p)\mapsto F^-(x,\eta,p)$ is defined by $H(x,F^-(x,\eta,p),p)=r\eta$, by the implicit function theorem, set $\xi=F^-(x,\eta,p)$,
we have
\begin{align*}
\dfrac{\partial}{\partial p}F^-(x,\eta,p)&~=~-\dfrac{H_p(x,\xi,p)}{H_\xi(x,\xi,p)}\\
&~=~\dfrac{\xi}{p}\cdot \dfrac{u^*(x,\xi,p)-x (\lambda +r)}{u^*(x,\xi,p)-x (\lambda +r)+xp(\lambda +\mu +v^*(x,\xi))}\\
&~=~\left(1+\dfrac{x(\lambda+\mu+v^*(x,\xi)}{H_\xi(x,\xi,p)}\right)\dfrac{\xi}{p}~>~\dfrac{F^-(x,\eta,p)}{p}~>~0.
\end{align*}
Since $q_{\varepsilon}(\cdot,x_0)$ is decreasing, it holds
\[
F^{-}\big(x_1,Z_{\varepsilon}(x_1,x_0),q_{\varepsilon}(x_1,x_0)\big)~\ge~ F^{-}\big(x_1,Z_{\varepsilon}(x_1,x_0),q_{\varepsilon}(x_2,x_0)\big),
\]
and \eqref{eq:cd1}-\eqref{eq:cd2} yield
\[
F^{-}\big(x_2,Z_{\varepsilon}(x_2,x_0),q_{\varepsilon}(x_2,x_0)\big)-F^{-}\big(x_1,Z_{\varepsilon}(x_1,x_0),q_{\varepsilon}(x_2,x_0)\big)~\ge~ \dfrac{\delta_{1,\flat}}{4}.
\]
On the other hand, from \eqref{eq:gradH} it follows that  the map $x\to F^-(x,\eta,p)$ is monotone decreasing and thus 
\begin{equation}\label{eq:F-11}
F^{-}\big(x_2,Z_{\varepsilon}(x_2,x_0),q_{\varepsilon}(x_2,x_0)\big)-F^{-}\big(x_2,Z_{\varepsilon}(x_1,x_0),q_{\varepsilon}(x_2,x_0)\big)~\ge~ \dfrac{\delta_{1,\flat}}{4}.
\end{equation}
Observe that the map $\eta\to F^-(x,\eta,p)$ is H\"{o}lder continuous due to Lemma \ref{lemma:holder}. More precisely, there exist a constant $C_{x^{\flat}}>0$ such that  
\[\left|F^-(x,\eta_2,p)-F^-(x,\eta_1,p)\right|~\le~ C_{x^{\flat}}\cdot \big|\eta_2-\eta_1\big|^{\frac{1}{2}}\]
for all $\eta_1,\eta_2\in (0,W(x)]$, $x\in[\bar{x},x^*]$, $p\in [\theta(x^*),1]$. Thus, \eqref{eq:F-11} implies that
\[
\displaystyle\left|Z_{\varepsilon}(x_2,x_0)-Z_{\varepsilon}(x_1,x_0)\right|~\ge~ \dfrac{\delta^2_{1,\flat}}{16C^2_{x^{\flat}}}.
\]
Recalling \eqref{eq:k-cod}, we have 
\[
Z'_{\varepsilon}(x,x_0)~=~F^{-}\left(x,Z_{\varepsilon}(x,x_0),q_{\varepsilon}(x,x_0)\right)~\le~ W'(x_0)+\dfrac{\delta_{1,x^{\flat}}}{2}\qquad\forall x\in [x_1,x_2],
\]
and this yields
\[
|x_2-x_1|~\ge~ \displaystyle \dfrac{\delta^2_{1,x^{\flat}}}{8C^2_{x^{\flat}}[2W'(x_0)+\delta_{1,x^{\flat}}]}.
\]
Therefore, 
\[
x_0-a_{\varepsilon}(x_0)~\ge~ \delta_{x^\flat}~:=~ \min~\left\{\delta_{1,x^{\flat}},\dfrac{\delta^2_{1,x^{\flat}}}{8C^2_{x^{\flat}}[2W'(x_0)+\delta_{1,x^{\flat}}]} \right\}~>~0,
\]
and the proof is complete.
\end{proof}
\medskip

\begin{remark} 
{\it In general, the backward Cauchy problem \eqref{eq:V-x0} may admit more than one solution.}
\end{remark}

As a consequence of Lemma \ref{lemma:bw}, there exists a sequence $\{\varepsilon_{n}\}_{n\geq 0}\to 0+$ 
such that  the sequence of backwards solutions $\{(Z_{\varepsilon_n}(\cdot,x_0),q_{\varepsilon_n}(\cdot,x_0))\}_{n\geq 1}$ converges to  $(Z(\cdot,x_0),q(\cdot,x_0))$ 
which is a solution of \eqref{eq:V-x0}. With the same argument in the proof of Proposition \ref{prop:pro1}, we can extend backward the solution $(Z(\cdot,x_0),q(\cdot,x_0))$ until $a(x_0)$ such that 
\[\lim_{x\to a(x_0)+}~Z(a(x_0),x_0)~=~W(a(x_0)),\]
and Lemma \ref{lemma:compW} yields $ \lim_{x\to a(x_0)+}q(a(x_0),x_0)\leq p_c(a(x_0))$. If the initial size of the debt is $\bar{x}\in [a(x_0),x_0]$ 
we think of $Z(\bar{x},x_0)$ is as the expected cost of \eqref{eq:cost-feed}-\eqref{eq:cont-feed} with $p(\cdot,x_0)$, $x(0)=x_0$ achieved by the feedback strategies
\begin{equation}\label{eq:u-xx}
\begin{cases}
u(x,x_0)&~=~\underset{w\in [0,1]}{\mathrm{argmin}}\left\{ L(w) - \,\dfrac{Z'(x,x_0)}{p(x,x_0)}\cdot w\right\},\cr\cr
v(x,x_0)&~=~\underset{v\geq 0}{\mathrm{argmin}}\Big\{c(v)-vxZ'(x,x_0)\Big\}\,.
\end{cases}
\end{equation}
With this strategy, the debt has the asymptotic behavior $x(t)\to x_0$ as $t\to\infty$.

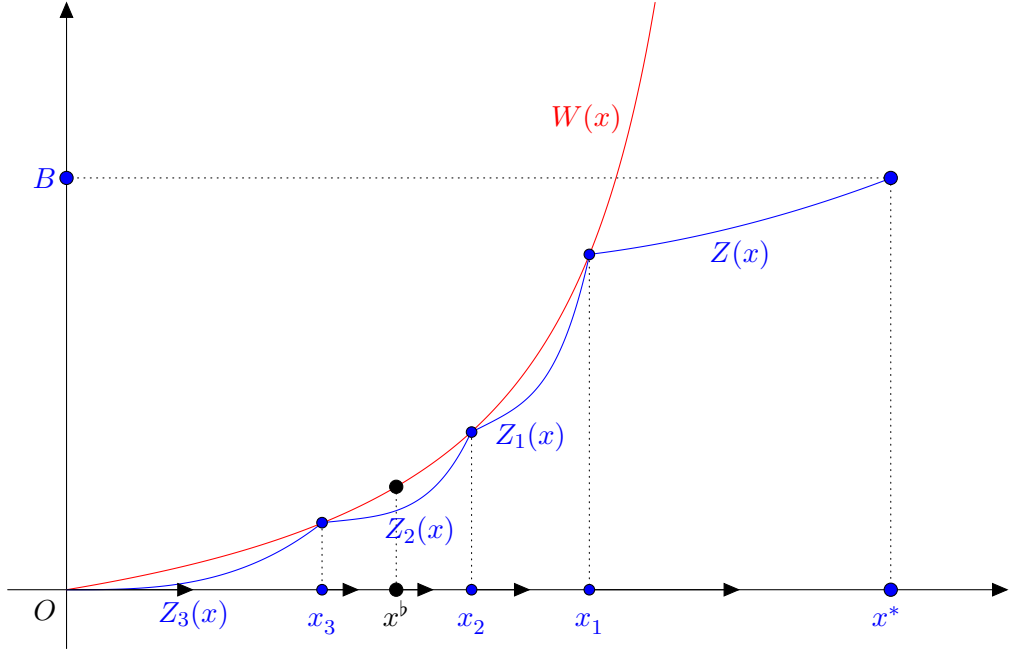
\begin{figure}[ht]
\begin{tikzpicture}[line cap=round,line join=round,>=triangle 45,x=0.1 \textwidth,y= 0.1 \textwidth]
\draw[->,color=black] (-0.5,0) -- (8,0);
\draw[->,color=black] (0.,-0.5) -- (0,5);
\clip(-0.5,-0.5) rectangle (8,5);
\draw[color=red,smooth,samples=100,domain=0:5] plot(\x,{6.0/(6.0-(\x))-1.0});
\draw[color=blue,smooth,samples=100,domain=0:2.17] plot(\x,{(\x)^(3.0)/18.0});
\draw[color=blue,smooth,samples=100,domain=2.17:3.44] plot(\x,{((\x)-2.17)^(4.0)/4.0+((\x)-2.17)/10.0+0.57});
\draw[color=blue,smooth,samples=100,domain=3.44:4.44] plot(\x,{((\x)-3.44)^(4.0)+((\x)-3.44)/2.0+1.34});
\draw[color=blue,smooth,samples=100,domain=4.44:7] plot(\x,{((\x) - 4.44)^2 / 20 + ((\x) - 4.44) / 8 + 2.85});

\draw[dotted] (2.17,0.57)--(2.17,0);
\draw[dotted] (2.80,0.875)--(2.80,0);
\draw[dotted] (3.44,1.34)--(3.44,0);
\draw[dotted] (4.44,2.85)--(4.44,0);
\draw[dotted] (7,3.5)--(7,0);
\draw[dotted] (7,3.5)--(0,3.5);

\draw[->] (0,0)--(1.08,0);\draw[->] (2.17,0)--(2.485,0);
\draw[->] (2.80,0)--(3.12,0);\draw[->] (3.44,0)--(3.94,0);
\draw[->] (4.44,0)--(5.72,0);

\node[color=black, anchor=north east] at (0,0) {$O$};

\draw[fill=blue] (2.17,0.57) circle (2.0pt);
\draw[fill=blue] (2.17,0) circle (2.0pt);
\node[color=blue, anchor = north] at (2.17,0) {$x_3^{\phantom{\flat}}$};

\draw[fill=blue] (3.44,1.34) circle (2.0pt);
\draw[fill=blue] (3.44,0) circle (2.0pt);
\node[color=blue, anchor = north] at (3.44,0) {$x_2^{\phantom{\flat}}$};

\draw[fill=blue] (4.44,2.85) circle (2.0pt);
\draw[fill=blue] (4.44,0) circle (2.0pt);
\node[color=blue, anchor = north] at (4.44,0) {$x_1^{\phantom{\flat}}$};

\draw[fill=blue] (7,3.5) circle (2.5pt);
\draw[fill=blue] (7,0) circle (2.5pt);
\draw[fill=blue] (0,3.5) circle (2.5pt);
\node[color=blue, anchor = north] at (7,0) {$x^{*\phantom{\flat}}$};
\node[color=blue, anchor = east] at (0,3.5) {$B$};

\draw[fill=black] (2.8,0.875) circle (2.5pt);
\draw[fill=black] (2.8,0) circle (2.5pt);
\node[color=black, anchor = north] at (2.8,0) {$x^{\flat}$};

\node[color=red, anchor=east] at (4.8,4) {$W(x)$};
\node[color=blue] at (5.72,2.85) {$Z(x)$};
\node[color=blue] at (3.94,1.30) {$Z_1(x)$};
\node[color=blue] at (3,0.50) {$Z_2(x)$};
\node[color=blue, anchor=north] at (1.08,0) {$Z_3(x)$};
\end{tikzpicture}
\label{fig:construction}\caption{Construction of a solution: starting from $(x^*,B)$ we solve backward the system until the first touch with the graph of $W$ at $(x_1,W(x_1))$.
Then we restart by solving backward the system with the new terminal conditions $(W(x_1),p_c(x_1))$, until the next touch with the graph of $W$ at $(x_2,W(x_2))$ and so on.
In a finite number of steps we reach the origin. If a touch occurs at $x_{n_0}<x^\flat$ then the backward solution from $x_{n_0}$ reaches the origin with $q\equiv 1$.
Given an initial value $\bar x$ of the DTI, if $0\le x_{n+1}<\bar x<x_n<x_1$ the the optimal strategy let the DTI increase asymptotically to $x_n$ (no banktuptcy), while
if $x_1<\bar x<x^*$ then the optimal strategy let the DTI increase to $x^*$, thus providing bankruptcy in finite time.}
\end{figure}
\subsubsection{Construction of an equilibrium solution.} 
We are now ready to construct an solution to the system of Hamilton-Jacobi equation \eqref{eq:Sdode} with boundary conditions \eqref{eq:boundary}. By induction, we define a family of back solutions as follows:
\[
x_1~:=~ x^*_{W}, \quad\qquad (Z_1(x),q_1(x))~=~\left(Z(x,x^*),q(x,x^*)\right)\qquad\forall x\in [x_1,x^*]
\]
and 
\[
x_{n+1}~:=~ a(x_n),\quad\qquad (Z(x,x_n),q(x,x_n))\qquad\forall x\in [x_{n+1},x_{n}]\,.
\]
From Case 1 and Lemma \ref{lemma:bw}, there exists a natural number $N_0<1+\dfrac{x^*-x^{\flat}}{\delta_{x^{\flat}}}$ 
such that our construction will be stop in $N_0$ step, i.e.,
\[
x_{N_0}~>~0,\qquad a(x_{N_0})~=~0\qquad\textrm{ and }\qquad \lim_{x\to a(x_{N_0})}Z(x,x_{N_0})~=~0\,.
\]
We will show that a feedback equilibrium solution to the debt management problem is obtained
as follows
\begin{equation}\label{eq:sol}
\left(V^*(x),p^*(x)\right)~=~\begin{cases}\left(Z(x,x^*),q(x,x^*)\right)\qquad\forall x\in (x_W,x^*],\\[4mm] 
\left(Z(x,x_{k}),q(x,x_k)\right) \qquad\forall x\in (a(x_k),x_k], k\in \{1,2,\dots, N_0\},\end{cases}\\[4mm],
\end{equation}
and
\begin{equation}\label{eq:sol-u}
\begin{cases}
u^*(x)&~=~\underset{w\in [0,1]}{\mathrm{argmin}}\left\{ L(w) - \dfrac{(V^*)'(x)}{p^*(x)}\cdot w\right\},\cr\cr
v^*(x)&~=~\underset{v\geq 0}{\mathrm{argmin}}\left\{c(v)-vx(V^*)'(x)\right\}.
\end{cases}
\end{equation}
\begin{proof}[Proof of Theorem \ref{thm:main2}]
From the monotone increasing property  of the maps  $\xi\mapsto v^*(x^*,\xi)$, $\eta\mapsto F^-(x^*,\eta,\theta(x^*))$ and $p\mapsto F^-(x^*,W(x^*), p)$, we have 
\begin{multline*}
\theta(x^*)\cdot (r+\lambda+v^*(x^*,F^-(x^*,B,\theta(x^*)))\\
~<~p_c(x^*)\cdot (r+\lambda+v^*(x^*,F^-(x^*,W(x^*),p_c(x^*)))~=~r+\lambda
\end{multline*}
and it yields \eqref{eq:Asp1}. By Proposition \ref{prop:pro1} and Lemma \ref{lemma:bw}, a pair $V^*(\cdot),p^*(\cdot)$ in \eqref{eq:sol} is well-defined on $[0,x^*]$. In the remaining steps,  we show that $V^*, p^*, u^*,v^*$ provide an equilibrium solution. Namely, they satisfy the properties (i)-(ii) 
in Definition \ref{def:eqsol}.\par\bigskip\par
\noindent \textbf{1.} To prove (i) in Definition \ref{def:eqsol}, let $V(\cdot)$ be the value function 
for the optimal control problem \eqref{eq:cost-feed}-\eqref{eq:cont-feed}. For any initial value,  $x(0)=x_0\in [0,x^*]$,  
the feedback controls $u^*$ and $v^*$ in  \eqref{eq:sol-u} yield the cost $V^*(x_0)$. This implies
\[V(x_0)~\leq ~V^*(x_0).\]  
To prove the converse inequality we need to show that, for any measurable control $u:[0,+\infty[\,\mapsto [0,1]$ and $v:[0,+\infty[\to [0,+\infty[$, calling $t\mapsto x(t)$ the solution to 
\begin{equation}\label{eq:ode7}
\dot{x}(t)~=~\left(\dfrac{\lambda+r}{p^*(x(t))} -\lambda -\mu-v(t)\right) x(t) - \dfrac{u(t)}{p^*(x(t))}\,,\qquad\qquad x(0)~=~x_0,
\end{equation}
it holds
\begin{equation}\label{eq:upVx_1}
\int_0^{T_b} e^{-rt}[L(u(x(t)))+c(v(x(t)))]\, dt + e^{-r T_b} B~\ge~ V^*(x_0)
\end{equation}
where 
\[T_b~=~\inf\, \bigl\{t\geq 0\,;~~x(t)=x^*\bigr\}\] is the bankruptcy time
(possibly with $T_b=+\infty$).
\par\bigskip\par 
For $t\in [0, T_b]$, consider the
absolutely continuous function 
\[\phi^{u,v}(t)~:=~ \int_{0}^te^{-rs}\cdot [L(u(s))+c(v(s))]~ds + e^{-rt}V^*
(x(t)).\]
At any Lebesgue point $t$ of $u(\cdot)$ and $v(\cdot)$, recalling that $(V^*,p^*)$ solves the system \eqref{eq:Sdode}, we compute
\begin{align*}
&\displaystyle \frac{d}{dt}\phi^{u,v}(t)~=~e^{-rt}\cdot \Big[L(u(t))+c(v(t))-rV^*(x(t))+
(V^*)'(x(t))\cdot\dot{x}(t)\Big]\\[4mm]
&~=\displaystyle ~e^{-rt}\cdot\Big[L(u(t))+c(v(t))-rV^*(x(t
))\\[4mm]
&\qquad\qquad\qquad+(V^*)'(x(t))\left(\left(\frac{\lambda+r}{p^*(x(t))}-\lambda-\mu-v(t)\right)x(t)-\frac{u(t)}{p^*(x(t))}\right)\Big]\\[4mm]
&\displaystyle\ge~ e^{-rt}\cdot\Big[ \min_{\omega\in[0,1]}
\left\{L(\omega)-\frac{(V^*)'(x(t))}{p^*(x(t))}\,\omega\right\} +\min_{\zeta\in [0+\infty[}\left\{c(\zeta)-(V^*)'(x(t))x(t)\,\zeta\right\}
\\[4mm]
&\qquad\qquad\qquad\qquad\qquad\qquad\qquad+ \left(\frac{\lambda+r}{p^*(x(t))}-\lambda-\mu\right)
x(t)(V^*)'(x(t))-rV^*(x(t))\Big]\\[3mm]
&~= ~e^{-rt}\cdot \Big[H\left(x(t),(V^*)'(x(t)), p^*(x(t))\right)-rV^*(x(t))\Big]= 0.
\end{align*}
Thus,
\[V^*(x_0)~=~\phi^{u,v}(0)\le  \lim_{t
\rightarrow T_b-
}\phi^{u,v}(t
)~=~\int_{0}^{T_b
} e^{-rt}\cdot [L(u(t))+c(v(t))]~dt + e^{-rT_b
}B,\]
and this yields  \eqref{eq:upVx_1}.
\par\bigskip\par
\noindent \textbf{2.} It remains to check (ii) in Definition \ref{def:eqsol}.  The case $x_0=0$ is trivial.
Two remain cases will be considered.
\par\bigskip\par
\noindent \emph{CASE 1:  If $x_0\in ]x_1,x^*]$} then $x(t)>x_1$ for all 
$t\in [0,T_b]$. This implies
\[
\dot{x}(t)~=~H_{\xi}(x(t),Z(x(t),x^*),q(x(t),x^*))\,.
\]
From the second equation in  \eqref{eq:Sdode} it follows 
\[\dfrac{d}{dt}~p(x(t))~=~p'(x(t))\dot x(t)~=~(r+\lambda+v^*(x(t)))p(x(t))-(r+\lambda),\]
Thus, for every $t\in [0, T_b]$ it holds
\[
p(x(0))~=~p(x(t))\cdot\int_{0}^{t} e^{-(r+\lambda+v^*(x(\tau)))}~d\tau+\int_{0}^t (r+\lambda)\int_{0}^{\tau} e^{-(r+\lambda+v^*(x(s)))}~ds~d\tau
\]
By letting $t\to T_b$, we obtain
\[
p(x_0)~=~\int_{0}^{T_b} (r+\lambda)\int_{0}^{\tau} e^{-(r+\lambda+v^*(x(s)))}~ds~d\tau+\theta(x^*)\cdot\int_{0}^{T_b} e^{-(r+\lambda+v^*(x(\tau)))}~d\tau
\]
\quad\\
 \emph{CASE 2:} Assume that $x_0\in \, [a(x_k),x_{k}[$ for some $k\in\{1,2,...,N_0\}$. In this case,  $T_b=+\infty$ and $x(t)\in [a_{x_k},x_k[$ such that 
\[
\lim_{t\to +\infty}~x(t)~=~x_{k}\,.
\]
With a similar computation, we obtain 
\[
p(x_0)~=~\theta(x^*)\cdot\int_{0}^{\infty} e^{-(r+\lambda+v^*(x(\tau)))}~d\tau
\]
proving (ii).
\end{proof}
\section{Dependence on $x^*$} In this section, we study the behavior of the total cost for servicing when the maximum size $x^*$ of the debt-ratio-income, at which bankruptcy is declared, becomes very large. It turns out that a crucial role in the asymptotic behavior of $V$ as $x^*\to +\infty$ is played by the \emph{speed of decay} of the salvage rate $\theta(x^*)$
as $x^*\to +\infty$, which represents the fraction of the investment that can be recovered by the investors after the bankruptcy (and the unitary bond discounted price at the bankruptcy threshold). 
More precisely, the following proposition show that
\begin{itemize}
\item if the salvage rate decay sufficiently slowly, i.e., the lenders can still recover a sufficiently high fraction of their investment after the bankruptcy, then the best choice for the borrower is to implement the Ponzi's scheme;
\item otherwise, if the salvage rate $\theta(x^*)$ decays sufficiently fast, then Ponzi's scheme is no longer an optimal solution for the borrower.
\end{itemize}
\begin{proposition}\label{Pz} Let $(V(x,x^*),p(x,x^*))$ be constructed in Theorem \ref{thm:main2}. The following holds:
\begin{itemize}
\item [(i)] if $\displaystyle\limsup_{s\to +\infty}\theta(s)s=R<+\infty$  then 
\begin{equation}\label{eq:bb}
\liminf_{x^*\to+\infty}V(x,x^*)~\ge~ B\cdot \left(1-\frac{R}{x}\right)^{\frac{r}{r+\lambda}}
\end{equation}
for all 
\[x~\ge~ \frac{1}{r-\mu}\cdot \max\left\{4,\frac{4B}{L'(0)}, \frac{4C_1B}{c'(0)}, 2C_1c^{-1}(rB)\right\}.\]
\item [(ii)] if $\displaystyle\lim_{s\to+\infty}~\theta(s)s=+\infty$ then 
\begin{equation}\label{eq:vvD}
\limsup_{x^*\to\infty}~V(x,x^*)~=~0\qquad\forall x\in [0,x^*[.
\end{equation}
\end{itemize}
\end{proposition}
\begin{proof} \textbf{1.} We first provide an upper bound on $v(\cdot,x^*)$. From \eqref{eq:Sdode} and \eqref{eq:detHamil}, we estimate 
\begin{align*}
H(x,\xi,p)&~\geq~\min_{v\geq 0}\left\{c(v)-x\xi v\right\}+[(r-\mu)x-1]\cdot \frac{\xi}{p}\\
&~\geq~\min_{v\geq 0}\left\{c(v)-x\xi v\right\}+	\frac{(r-\mu)x}{2}\cdot \frac{\xi}{p}~:=~ K(x,\xi,p)
\end{align*}
for all $\xi,p >0$ and $x\geq \dfrac{2}{r-\mu}$. We compute 
\[K_{\xi}(x,\xi,p)~=~\frac{(r-\mu)x}{2p}-xv_K\]
where 
\[v_{K}~=~\begin{cases}0&\textrm{ if }~~~0\le x\xi< c'(0),\cr\cr (c')^{-1}(x\xi)&\textrm{ if }~~~x\xi\ge c'(0)>0.\end{cases}\]
This implies that the maximum of $K$ is achieved for $v_K=\displaystyle\frac{r-\mu}{2p}$ and its value is 
\[
\max_{\xi\geq 0}~K(x,\xi,p)~=~K(x,\xi_K,p)~=~c\left(\frac{r-\mu}{2p}\right),\textrm{ with }\xi_K~=~\frac{c'(v_K)}{x}\,.
\]
Thus, the monotone increasing property of the map $\xi\to H(x,\xi,p(x,x^*))$ on the interval\newline $\big[0,\xi^{\sharp}(x,p(x,x^*))\big]$ implies that 
\begin{equation}\label{eq:v-cd}
F^-(x,V(x,x^*),p(x,x^*))~<~\xi_K\qquad\mathrm{\Longrightarrow}\qquad v(x,x^*)~\le~ \frac{r-\mu}{2p(x,x^*)}.
\end{equation}
provided that $\displaystyle c\left(\frac{r-\mu}{2p(x,x^*)}\right)\geq rB$. From \eqref{eq:Sdode}) and \eqref{eq:detHamil}, it follows 
\begin{multline*}
rB~\geq~-x V'(x,x^*)v(x,x^*)+[(r-\mu)x-u(x,x^*)]\cdot \frac{V'(x,x^*)}{p(x,x^*)}\\
~\geq~ \left[\frac{(r-\mu)x}{2}-1\right]\cdot \frac{V'(x,x^*)}{p(x,x^*)}~\ge~ \frac{(r-\mu)x}{4}\cdot \frac{V'(x,x^*)}{p(x,x^*)}.
\end{multline*}
Thus,  if 
\begin{equation}\label{eq:cod11}
p(x,x^*)~\le~  \min\left\{\frac{r-\mu}{2c^{-1}(rB)}, \frac{(r-\mu)c'(0)}{4B}\right\}\quad\textrm{ and }\quad x~\ge~  \max\left\{\frac{4}{r-\mu},\frac{4B}{(r-\mu)L'(0)}\right\}
\end{equation}
then 
\begin{equation}\label{eq:u=0}
\frac{V'(x,x^*)}{p(x,x^*)}~\le~  \frac{4B}{(r-\mu )x}~\le~ L'(0)\qquad\mathrm{\Longrightarrow}\qquad u(x,x^*)~=~0,
\end{equation}
and 
\begin{equation}\label{eq:v=0-D}
V'(x,x^*)x~\le~ \frac{4B}{r-\mu}\cdot p(x,x^*)~\le~ c'(0)\qquad\mathrm{\Longrightarrow}\qquad v(x,x^*)~=~0.
\end{equation}
In this case, from \eqref{eq:Sdode}, \eqref{eq:detHamil} and \eqref{eq:gradH}, it holds
\[
(r+\lambda)(p(x,x^*)-1)~=~\left(\dfrac{\lambda+r}{p(x,x^*)}-\lambda-\mu\right)xp'(x,x^*)\,.
\]
Thus, 
\[
p(x,x^*)~=~\dfrac{\theta(x^*) x^*}{x}\cdot \left(\dfrac{1-p(x,x^*)}{1-\theta(x^*)}\right)^{\frac{r-\mu}{r+\lambda}}
\]
provided that \eqref{eq:cod11} holds.
\medskip

\textbf{2.} Assume that 
\[
\limsup_{s\in [0,+\infty)}~\theta(s)s~=~R~<~+\infty,
\]
there exists a constant $C_1<+\infty$ such that $\sup_{s\in [0,+\infty)}~\theta(s)s=C_1$. Since $p(\cdot,x^*)$ is increasing, it holds
\begin{equation}\label{eq:upb-p}
p(x,x^*)~\le~ \dfrac{\theta(x^*) x^*}{x}~\le~ \frac{C_1}{x}\qquad\textrm{if \eqref{eq:cod11} holds}\,.
\end{equation}
Denote by 
\[M~:=~\dfrac{1}{r-\mu}\cdot \max\left\{4,\frac{4B}{L'(0)}, \frac{4C_1B}{c'(0)}, 2C_1c^{-1}(rB)\right\},\]
we then have
\[u(x,x^*)~=~v(x,x^*)~=~0\qquad\forall x\in [M,x^*], x^*\geq M\,.\]
From \eqref{eq:Sdode}, \eqref{eq:detHamil} and \eqref{eq:gradH}, $(V,p)$ solves the system of ODEs 
\begin{equation}\label{eq:ode---}
\begin{cases}
V'(x,x^*)&~=~\dfrac{rp}{[(\lambda+r)-(\lambda+\mu)p(x,x^*)]x}\cdot V\\[4mm]
p'(x,x^*)&~=~(\lambda+r)\cdot \dfrac{p(x,x^*)(p(x,x^*)-1)}{[(\lambda+r)-(\lambda+\mu)p(x,x^*)]\, x }
\end{cases} 
\end{equation}
for all  $x\in [M,x^*]$ with $x^*\geq M$. Solving the above system of ODEs (see in Section 5 of \cite{BMNP}), we obtain that 
\[
V(x,x^*)~=~B\cdot  \left(\dfrac{1-p(x,x^*)}{1-\theta(x^*)}\right)^{\frac{r}{r+\lambda}},\quad\qquad p(x,x^*)~=~\dfrac{\theta(x^*) x^*}{x}\cdot \left(\dfrac{1-p(x,x^*)}{1-\theta(x^*)}\right)^{\frac{r-\mu}{r+\lambda}}
\]
for all $x\geq [M,x^*]$. Thus,
\[
\liminf_{x^*\to+\infty}V(x,x^*)~\ge~ B\cdot \left(1-\frac{R}{x}\right)^{\frac{r}{r+\lambda}}\qquad\forall x\geq M
\]
and this yields \eqref{eq:bb}.
\medskip

\textbf{3.} We are now going to prove (ii). Assume that 
\begin{equation}\label{eq:as131}
\limsup_{s\to+\infty}~\theta(s)s~=~+\infty\,.
\end{equation}
Set 
\[\gamma~:=~ \min\left\{\frac{r-\mu}{2c^{-1}(rB)}, \frac{(r-\mu)c'(0)}{4B}\right\}\qquad\textrm{ and }\qquad M_2~:=~ \max\left\{\frac{4}{r-\mu},\frac{4B}{(r-\mu)L'(0)}\right\}\,.\]
For any $x^*>M_2$, denote by 
\[
\tau(x^*)~:=~ \begin{cases}
x^*\qquad&\mathrm{if}\qquad \theta(x^*)\geq \gamma\,,\\[4mm]
\inf\left\{x\geq M_2~\Big|~p(x,x^*)\leq \gamma \right\} \qquad&\mathrm{if}\qquad\ \theta(x^*) < \gamma\,.
\end{cases} 
\]
From \eqref{eq:cod11}--\eqref{eq:v=0-D}, the decreasing property of $p$ yields
\begin{equation}\label{eq:upp}
p(x,x^*)~\ge~ \gamma \qquad \forall x\in [M_2,\tau(x^*)[
\end{equation}
and
\[
p(x,x^*)~<~\gamma\qquad\Longrightarrow \qquad u(x,x^*)~=~v(x,x^*)\qquad \qquad\forall x\in [\tau(x^*),x^*]\,.
\]
As in the step 2, for any $x\in [\tau(x^*),x^*]$, we have 
\[
V(x,x^*)~=~B\cdot  \left(\dfrac{1-p(x,x^*)}{1-\theta(x^*)}\right)^{\frac{r}{r+\lambda}},\quad\qquad p(x,x^*)~=~\dfrac{\theta(x^*) x^*}{x}\cdot \left(\dfrac{1-p(x,x^*)}{1-\theta(x^*)}\right)^{\frac{r-\mu}{r+\lambda}}
\]
This implies that 
\begin{equation}\label{eq:bv-e1}
V(x,x^*)~=~B\cdot \left(\frac{p(x,x^*)x}{\theta(x^*)x^*}\right)^{\frac{r}{r-\mu}}~\le~ B\cdot \left(\frac{x}{\theta(x^*)x^*}\right)^{\frac{r}{r-\mu}}
\end{equation}
for all $x\in [\tau(x^*),x^*]$. On the other hand, for any $x\in [M_2,\tau(x^*)]$, from \eqref{eq:Sdode}, \eqref{eq:detHamil} and \eqref{eq:upp}, it holds 
\[
rV(x,x^*)~\le~ \frac{r+\lambda}{p(x,x^*)}xV'(x,x^*)~\le~ \frac{(r+\lambda)x}{\gamma}\cdot V'(x,x^*)\,.
\]
This implies that 
\begin{equation}\label{eq:bv-e2}
V(x,x^*)\le V(\tau(x^*),x^*)\cdot \left(\frac{x}{\tau(x^*)}\right)^{\frac{r\gamma}{r+\lambda}}~\le~ B\cdot\left(\frac{x}{\tau(x^*)}\right)^{\frac{r\gamma}{r+\lambda}}\qquad\forall  x\in [M_2,\tau(x^*)].
\end{equation}
For any fix $x_0\geq M_2$, we will prove  that 
\begin{equation}\label{eq:pr}
\limsup_{x^*\to+\infty}~V(x_0,x^*)~=~0.
\end{equation}
Two cases are considered:
\begin{itemize}
\item If $\limsup_{x^*\to+\infty}\tau(x^*)=+\infty$ then \eqref{eq:bv-e2} yields
\[
\lim_{x^*\to+\infty}V(x_0,x^*)~\le~ \liminf_{x^*\to+\infty}~B\cdot \left(\frac{x_0}{\tau(x^*)}\right)^{\frac{r\gamma}{r+\lambda}}~=~0.
\]
\item If $\limsup_{x^*\to+\infty}\tau(x^*)<+\infty$ then 
\[
\tau(x^*)~<~M_3\qquad\forall x^*~>~0
\]
for some $M_3>0$. Recalling \eqref{eq:bv-e1} and \eqref{eq:as131}, we obtain that 
\[
\lim_{x^*\to\infty}~V(x_0,x^*)~\le~ \lim_{x^*\to\infty}V(x_0+M_3,x^*)~\le~ \lim_{x^*\to\infty}B\cdot \left(\frac{x_0+M_3}{\theta(x^*)x^*}\right)^{\frac{r}{r-\mu}}~=~0.
\]
\end{itemize}
Thus, \eqref{eq:pr} holds and the increasing property of $V(\cdot,x^*)$ yields \eqref{eq:vvD}.
\end{proof}

We complete this section by showing that for sufficiently large initial debt-ratio-income and bankruptcy threshold and recovery fraction after bankruptcy, the optimal strategy for the borrower will use currency devaluation to deflate the debt-ratio-income. For simplicity, let us consider $x^*$ and $B^*$ sufficiently large such that

\begin{equation}\label{eq:cond12}
x^*~>~\dfrac{L'(0)+Br}{L'(0)\cdot (r-\mu)}\qquad\textrm{ and }\qquad B~\geq~\dfrac{2(r-\mu)c'(0)}{r}.
\end{equation}
In this case, the following holds:
\medskip

\begin{proposition}[Devaluating strategies]
Let $x\mapsto \left(V(x,x^*),p(x,x^*)\right)$ be an equilibrium solution of \eqref{eq:Sdode} with boundary conditions \eqref{eq:boundary}. If 
\begin{equation}\label{eq:cond13}
\theta(x^*)x^*~>~\dfrac{2(r+\lambda)c'(0)}{r-\mu}\cdot\left(\dfrac{1}{rB}+\dfrac{1}{L'(0)}\right)
\end{equation}
then the function
\[v^*(x,x^*)~=~\underset{\omega\ge 0}{\mathrm{argmin}}\left\{c(\omega)-\omega x V'(x,x^*)\right\}\]
is not identically zero.
\end{proposition}
\begin{proof}
Set $M:=\dfrac{L'(0)+Br}{L'(0)\cdot (r-\mu)}$. Assume by a contradiction that $v^*(x,x^*)=0$ for all $x\in [M,x^*]$. In particular, we have 
\begin{equation}\label{eq:aspL}
0~\le~ x V'(x,x^*)~\le~ c'(0)\qquad x\in [M,x^*].
\end{equation}
The system \eqref{eq:Sdode} in $[M,x^*]$ reduces to 
\begin{equation}\label{eq:redSode}
\begin{cases}
rV(x)~=~\tilde{H}(x,V'(x),p(x))\cr\cr
(r+\lambda)(p(x)-1)~=~\tilde H_{\xi}(x,V'(x),p(x))\cdot p'(x)
\end{cases}\end{equation}
with
\[
\tilde{H}(x,\xi,p)~=~\min_{u\in [0,1]}\left\{ L(u) - \dfrac{u}{p}\, \xi \right\}+ \left(\dfrac{\lambda+r}{p} -\lambda-\mu\right) x\, \xi.
\]
Since  $r> \mu$ and $p\in [0,1]$, it holds
\begin{align*}
\tilde{H}(x,\xi,p)~\ge~-\dfrac{\xi}{p}+ \left(\lambda+r-p(\lambda+\mu)\right) x\, \dfrac{\xi}{p}
~\ge~\left((r-\mu)x-1\right)\cdot \dfrac{\xi}{p}
\end{align*}
and  \eqref{eq:redSode} yields  
\[
rB~\ge~rV(x,x^*)~\ge~\left((r-\mu)x-1\right)\cdot\dfrac{V'(x,x^*)}{p(x,x^*)}.
\]
Thus, for $x\in \left[M,x^*\right]$ we obtain 
\[
\dfrac{V'(x,x^*)}{p(x,x^*)}~\le~ \dfrac{rB}{(r-\mu)x-1}\le L'(0),
\]
which immediately implies 
\[
u^*(x,x^*)~:=~\underset{u\in [0,1]}{\mathrm{argmin}}\left\{ L(u) - u\cdot\dfrac{V'(x,x^*)}{p(x,x^*)}\right\}~=~0.
\]
Hence, $(V(\cdot,x^*),p(\cdot,x^*))$ solves (\ref{eq:Sdode}) on $[M,x^*]$ and 
\begin{align}\label{eq:exp-odes}
V(x,x^*)&~=~B\cdot\left(\dfrac{1-p(x,x^*)}{1-\theta(x^*)}\right)^{\frac{r}{r+\lambda}}~\geq~ B\cdot  \left(1-\dfrac{r}{r+\lambda}\cdot p(x,x^*)\right),\\ \nonumber
p(x,x^*)&~=~\dfrac{\theta(x^*) x^*}{x}\cdot \left(\dfrac{1-p(x,x^*)}{1-\theta(x^*)}\right)^{\frac{r-\mu}{r+\lambda}}~\geq~\dfrac{\theta(x^*) x^*}{x}\cdot  \left(1-\dfrac{r-\mu}{r+\lambda}\cdot p(x,x^*)\right),
\end{align}
for all $x\in [M,x^*]$. From the above inequality, we derive 
\[
p(x,x^*)~\geq~\dfrac{(r+\lambda)\theta(x^*)x^*}{(r+\lambda )x+(r-\mu)\theta(x^*)x^*}.
\]
Thus, \eqref{eq:aspL} and the first equation in \eqref{eq:exp-odes} imply 
\begin{align*}
c'(0)&~\geq~xV'(x,x^*)~=~rp(x,x^*)\cdot \dfrac{V(x,x^*)}{(\lambda+r)-(\lambda+\mu)p(x,x^*)}\\
&~\geq~\dfrac{rp(x,x^*)B}{r+\lambda } \cdot \dfrac{r+\lambda -r p(x,x^*)}{(\lambda+r)-(\lambda+\mu)p(x,x^*)}~\geq~\dfrac{rp(x,x^*)B}{r+\lambda}~\geq~\dfrac{rB\theta(x^*)x^*}{(r+\lambda )x+(r-\mu)\theta(x^*)x^*}
\end{align*}
for all $x\in [M,x^*]$. In particular, choose $x=M$ and recall \eqref{eq:cond12}, we get 
\[
M~\geq~ \dfrac{rB-(r-\mu)c'(0)}{(r+\lambda) c'(0)}\cdot {\theta(x^*)x^*}~\geq~\dfrac{rB}{2(r+\lambda)c'(0)}\cdot {\theta(x^*)x^*}
\]
and it contradicts \eqref{eq:cond13}.
\end{proof}
\section{Appendix}
We first introduce now some concepts of convex analysis, referring the reader to \cite{ET} and \cite{R} for a comprehensive introduction to the subject.
\begin{definition}[Convex conjugate and subdifferential]
We recall that the convex conjugate $F^\circ:\mathbb R^d\to\mathbb R\cup \{\pm\infty\}$ of a map $F:\mathbb R^d\to\mathbb R\cup\{+\infty\}$
is the lower semicontinuous convex function defined by 
\[F^\circ(z^*)~=~\sup_{z\in\mathbb R^d}\Big\{\langle z^*,z\rangle-F(z)\Big\}.\]
Let $F:\mathbb R^d\to\mathbb R\cup\{+\infty\}$ be proper (i.e., not identically $+\infty$), convex, lower semicontinuous functions, 
$x\in\mathrm{dom}\, F:=\{x\in\mathbb R^d:\,F(x)\in\mathbb R\}$. 
We define the \emph{subdifferential in the sense of convex analysis} of $F$ at $x$ by setting
\[\partial F(x)~:=~\{v_x\in \mathbb R^d:\,F(y)-F(x)\ge \langle v_x,y-x\rangle\textrm{ for all }y\in \mathbb R^d\}.\]
\end{definition}

The following result provide a list of some properties of the sub-differential in the sense of convex analysis.
\medskip

\begin{lemma}[Properties of the subdifferential]\label{lemma:subdifprop}
Let $F,G:\mathbb R^d\to\mathbb R\cup\{+\infty\}$ be proper (i.e., not identically $+\infty$), convex, lower semicontinuous functions, 
\begin{enumerate}
\item If $F$ is classically (Fr\'echet) differentiable at $x$, then $\partial F(x)=\{F'(x)\}$.
\item $z^*\in \partial F(z)$ if and only if $z\in \partial F^\circ(z^*)$.  
\item $F(x_0)=\displaystyle\min_{x\in \mathbb R^d} F(x)$ if and only if $0\in \partial F(x_0)$
\item $z^*\in \partial F^\circ(z)$ if and only if $F(z)+F^\circ(z^*)=\langle z^*,z\rangle$. In this case $z^*\in\mathrm{dom}\,F^\circ$;
\item $\lambda\ge 0$ we have $\partial(\lambda F)(z)=\lambda \partial F(z)$;
\item if there exists $z\in\mathrm{dom}(F)\cap \mathrm{dom}(G)$ such that $F$ is continuous at $z$ then
$\partial (F + G)(x) = \partial F (x) + \partial G(x)$ for all $x\in\mathrm{dom}(F)\cap \mathrm{dom}(G)$;
\item let $\bar y\in \mathbb R^m$, $\Lambda:\mathbb R^m\to \mathbb R^d$ be a linear map, $G$ be continuous and finite at $\Lambda(\bar y)$; 
Then $\partial(G\circ\Lambda)(y)=\Lambda^T\partial G(\Lambda y)$ for all $y\in \mathbb R^m$, where  $\Lambda^T:\mathbb R^d\to\mathbb R^m$
is the adjoint of $\Lambda$.
\end{enumerate}
\end{lemma}

We now collect some technical results related to the Hamiltonian function:

\begin{lemma}\label{lemma:subdiffcost} If  \textbf{(A1)}-\textbf{(A2)} hold then $L^\circ,c^\circ:\mathbb R\to\mathbb R$ are continuously differentiable such that 
\begin{align*}L^\circ(\rho)\le\max\{0,\rho\},\qquad\quad c^\circ(\rho)\le\max\{0,v_{\max}\rho\},\end{align*}
and 
\begin{align*}
(L^\circ)'(\rho)~=~\begin{cases}0,&\textrm{ if }\rho<L'(0),\\\\ (L')^{-1}(\rho),&\textrm{ if }\rho\ge L'(0),\end{cases}&&
(c^\circ)'(\rho)~=~\begin{cases}0,&\textrm{ if }\rho<c'(0),\\\\ (c')^{-1}(\rho),&\textrm{ if }\rho\ge c'(0).\end{cases}
\end{align*}
\end{lemma}
\begin{proof}
Recalling the assumptions $\textbf{(A1)}-\textbf{(A2)}$ on $L,c$, the equations
\begin{align*}
L^\circ(\rho_1)+L(u)~=~u\rho_1,&&c^\circ(\rho_2)+c(v)~=~v\rho_2,
\end{align*}
admits as unique solutions 
\begin{align*}
u(\rho_1)~=~\begin{cases}0,&\textrm{ if }\rho_1<L'(0),\\ \\(L')^{-1}(\rho_1),&\textrm{ if }\rho_1\ge L'(0)>0,\end{cases}&&
v(\rho_2)~=~\begin{cases}0,&\textrm{ if }\rho_2<c'(0),\\ \\(c')^{-1}(\rho_2),&\textrm{ if }\rho_2\ge c'(0)\ge 0.\end{cases}
\end{align*}
The result now follows from Theorem 23.5, Theorem 25.1, and Theorem 26.3 in \cite{R}.
For the second part, set $I_C(s)=0$ if $s\in C$ and $0$ otherwise, 
since $L(u)\ge I_{[0,1]}(u)$ and $c(v)\ge I_{[0,v_{\max}]}(v)$, we have
\begin{align*}
L^\circ(\rho)~\le~ I^\circ_{[0,1]}(\rho)&~=~\max_{u\in [0,1]}\langle u,\rho\rangle=\max\{0,\rho\},\\
c^\circ(\rho)~\le~ I^\circ_{[0,v_{\max}]}(\rho)&~=~\max_{v\in [0,v_{\max}]}\langle v,\rho\rangle~=~\max\{0,v_{\max}\cdot\rho\}
\end{align*}
and this complete the proof.
\end{proof}

\par\medskip\par

As a consequence of Lemma \ref{lemma:subdiffcost}, the following holds:

\begin{lemma}\label{lemma:Hamgrad}
Assume \textbf{(A1)}-\textbf{(A2)}, and let $H$ be defined as in \eqref{eq:detHamil}. Then $H$ is continuous differentiable and
its gradient at points $(x,\xi,p)\in [0,+\infty[\times[0,+\infty[\times]0,1]$ can be expressed in terms of 
$u^*(\xi,p):=(L^\circ)'(\xi/p)$ and $v^*(x,\xi):=(c^\circ)'(x\xi)$ by 
\begin{equation}\begin{cases}
\displaystyle H_{x}(x,\xi,p)&~=~ \displaystyle \Big[(\lambda+r)-p(\lambda +\mu +v^*(x,\xi))\Big]\cdot \frac{\xi}{p},\\
\displaystyle H_{\xi}(x,\xi,p)&~=~ \displaystyle \frac{1}{p}\cdot \Big[x\big((\lambda+r)-p(\lambda +\mu +v^*(x,\xi))\big)-u^*(\xi,p)\Big],\\
\displaystyle H_{p}(x,\xi,p)&~=~ \displaystyle (u^*(\xi,p)-x(\lambda +r))\cdot \frac{\xi}{p^2},
\end{cases}
\end{equation}
and
\[\begin{cases}
u^*(\xi,p)&~=~\underset{u\in [0,1]}{\mathrm{argmin}}\left\{ L(u) - u\,\dfrac{\xi}{p}\right\},\\ \\
v^*(x,\xi)&~=~\underset{v\ge 0}{\mathrm{argmin}}\left\{ c(v) - vx\xi\right\}.
\end{cases}\]
Moreover, for all $x>0$, $0<p\le 1$, it holds 
\begin{align}
\label{eq:optuv}&\nabla u^*(\xi,p)~=~\dfrac{(1,-L'(u^*(x,\xi,p)))}{pL''(u_*(x,\xi,p))}\qquad\textrm{ if }~~~\xi> pL'(0),\\
\nonumber &\nabla v^*(x,\xi)~=~\dfrac{(\xi,x)}{c''(v^*(x,\xi))}\qquad \textrm{ if }~~~x\xi> c'(0),\\
\nonumber &\lim_{\xi\to +\infty}v^*(x,\xi)~=~v_{\max}.
\end{align}
\end{lemma}

\begin{lemma}\label{lemma:Hprop}
Let the assumptions \textbf{(A1)}-{\bf(A2)} hold. Then 
\begin{enumerate}
\item for all $\xi\geq 0$ and $p\in ]0,1]$,  the function $H$ satisfies
\begin{align*}
H(x,\xi,p)&~\le~\left(\dfrac{\lambda+r}{p} -(\lambda+\mu) \right) x\xi;\\
H(x,\xi,p)&~\ge~\left(\frac{(\lambda+r) x-1}{p}-(\lambda+\mu+v^*(x,\xi))x\right)\cdot \xi\\
&~\ge~\left(\frac{(\lambda+r) x-1}{p}-(\lambda+\mu+v_{\max})x\right)\cdot \xi;\\
H_\xi(x,\xi,p)&~\le~\left(\dfrac{\lambda+r}{p} -(\lambda+\mu)\right) x;\\
H_\xi(x,\xi,p)&~\ge~\frac{(\lambda+r) x-1}{p}-(\lambda+\mu+v^*(x,\xi))x\\
&~\ge~\frac{(\lambda+r)x-1}{p}-(\lambda+\mu+v_{\max})x;
\end{align*}
\item for every $x,p>0$ 
the map $\xi\mapsto H(x,\xi,p)$ is concave down and satisfies
\begin{align*}
H(x,0,p)~=~0,&&H_\xi(x,0,p)~=~\left(\dfrac{\lambda+r}{p} -(\lambda+\mu)\right) x.
\end{align*}
\end{enumerate}
\end{lemma}
\begin{proof}
The concavity of $\xi\mapsto H(x,\xi,p)$ for every $x,p>0$ is immediate from the definition of $H$ in \eqref{eq:detHamil}. 
The equalities in item $(2)$ are immediate from Lemma \ref{lemma:subdiffcost}.
The upper bound on $H(x,\xi,p)$ follows from the positivity of $L^\circ$ and $c^\circ$.
By concavity, the map $\xi\mapsto H_\xi(x,\xi,p)$ is monotone decreasing, thus $H_\xi(x,\xi,p)\le H_\xi(x,0,p)$, which proves the upper bound
on $H_\xi(x,\xi,p)$ together with item (2). The lower estimate for $H(x,\xi,p)$ comes from the second part of Lemma \ref{lemma:subdiffcost},
in particular from the upper estimate on $L^\circ(\cdot)$.
The lower estimate for $H_\xi(x,\xi,p)$ comes from Lemma \ref{lemma:Hamgrad}, noticing that
\[\lim_{\xi\to +\infty}u^*(\xi,p)~=~\lim_{\rho\to +\infty}(L')^{-1}(\rho)=1,\qquad \lim_{\xi\to +\infty}v^*(x,\xi)~=~\lim_{\rho\to +\infty}(c')^{-1}(\rho)~=~v_{\max},\]
and using the decreasing property of $\xi\mapsto H_\xi(x,\xi,p)$, i.e., the fact that
\[\lim_{\zeta\to +\infty}H_\xi(x,\zeta,p)~\le~ H_\xi(x,\xi,p)\]
for all $x\ge 0$, $p\in ]0,1]$, $\xi\in\mathbb R$.
\end{proof}
\medskip

\begin{lemma}\label{lemma:holder-gen}
Assume that $f:I\to\mathbb R$ is a $C^2$ convex strictly increasing function defined on a real interval $I$, and satisfying $f''\ge \delta >0$.
Then, denoted by $g$ its inverse function, $g:f(I)\to I$, we have that $g$ is $1/2$-H\"older continuous.
\end{lemma}
\begin{proof}
Indeed, let $x_1,x_2\in f^{-1}(I)$ with $x_1\le x_2$, and set $y_1=g(x_1)$ and $y_2=g(x_2)$.
\begin{align*}
f(y_2)-f(y_1)&~=~\int_{y_1}^{y_2} f'(t)\,dt=\int_{y_1}^{y_2} [f'(t)-f'(y_1)]\,dt\\ &~=~f'(y_1) \cdot (y_2-y_1)+\int_{y_1}^{y_2}\int_{y_1}^t f''(s)\,ds\,dt\\ 
&~\ge~f'(y_1) \cdot (y_2-y_1)+\dfrac{\delta }{2} (y_2-y_1)^2\ge \dfrac{\delta }{2} (y_2-y_1)^2,
\end{align*}
since $f$ is strictly increasing, $f''(s)\ge \delta $, and $y_1\le y_2$.
Thus if $x_2\ge x_1$ we have
\[|g(x_2)-g(x_1)|~\le~ \sqrt{\dfrac{2}{\delta }}|x_2-x_1|^{1/2}.\]
By switching the roles of $x_2$ and $x_1$, the same holds true if $x_1\ge x_2$.
\end{proof}
\v

{\bf Acknowledgments.} The research by K. T. Nguyen was partially supported by a grant from
the Simons Foundation/SFARI (521811, NTK).

\end{document}